\documentclass[suppldata]{interact}

\usepackage{epstopdf}
\usepackage[caption=false]{subfig}
\usepackage{hyperref}
\usepackage[dvipsnames]{xcolor}
\usepackage{stackengine}

\usepackage{multirow}
\usepackage{thmtools}

\usepackage[numbers,sort&compress]{natbib}
\bibpunct[, ]{[}{]}{,}{n}{,}{,}
\renewcommand\bibfont{\fontsize{10}{12}\selectfont}

\theoremstyle{plain}
\newtheorem{theorem}{Theorem}[section]

\newtheorem{corollary}[theorem]{Corollary}

\theoremstyle{definition}
\newtheorem{definition}[theorem]{Definition}
\newtheorem{example}[theorem]{Example}

\renewcommand\thmcontinues[1]{Continued}

\theoremstyle{remark}
\newtheorem{remark}{Remark}

\newcommand{\edited}[1]{{#1}}

\begin{document}


\title{Optimal inexactness schedules for Tunable Oracle based Methods}

\author{
\name{Guillaume Van Dessel\textsuperscript{a}\thanks{CONTACT Guillaume Van Dessel. Email: guillaume.vandessel@uclouvain.be} and François Glineur\textsuperscript{a,b}}
\affil{\textsuperscript{a}UCLouvain, ICTEAM (INMA), 4 Avenue Georges Lemaître, Louvain-la-Neuve, BE; \\\textsuperscript{b}UCLouvain, CORE, 34
Voie du Roman Pays, Louvain-la-Neuve, BE}
}

\maketitle

\begin{abstract}
Several recent works address the impact of inexact oracles in the convergence analysis of modern first-order optimization techniques, e.g.\@ Bregman Proximal Gradient and Prox-Linear methods as well as their accelerated variants, extending their field of applicability. In this paper, we consider situations where the oracle's inexactness can be chosen upon demand, more precision coming at a computational price counterpart. Our main motivations arise from oracles requiring the solving of auxiliary subproblems or the inexact computation of involved quantities, e.g.\@ a mini-batch stochastic gradient as a full-gradient estimate. We propose optimal inexactness schedules according to presumed oracle cost models and patterns of worst-case guarantees, covering among others convergence results of the aforementioned methods under the presence of inexactness.
Specifically, we detail how to choose the level of inexactness at each iteration to obtain the best trade-off between convergence and computational investments. Furthermore, we highlight the benefits one can expect by tuning those oracles' quality instead of keeping it constant throughout. Finally, we provide \edited{extensive} numerical experiments that support the practical interest of our approach, both in \emph{offline} and \emph{online} settings, applied to the \edited{Fast Gradient} algorithm.
\end{abstract}


\section{Introduction}
\label{intro_notes}
Typical iterative optimization schemes rely on key ingredients often referred to as \emph{oracles} \cite{HT20}. In what concerns continuous optimization, with respect to both convex and nonconvex realms, the vast majority of the papers use two distinctive types of oracles at each iteration, namely the \emph{informative} and \emph{computational} ones \cite{Thekum20}. 
\begin{enumerate}
    \item[(I)] \textbf{Informative}: one assumes the possibility to obtain differential information about the objective (zero, first, $\dots$, higher-order) at successive query points.
    \item[(II)] \textbf{Computational}: one assumes the ability to update sequences of iterates following some rules, often involving the resolution of easy or, at least, not too complicated, subproblems.
\end{enumerate}
By default, one considers implicitly that both oracles yield \emph{error-free} outputs. Nevertheless there exist cases for which such commodity appears unreasonable, e.g.\@ when the objective value stands as the result of a non-trivial optimization problem (I) \cite{Zhang21} or when one cannot solve exactly subproblems in (II) \cite{Mokh18}. Unfortunately, more and more problems of practical interest exhibit a structure that does not allow for \emph{exact} oracles. Therefore there has been put a lot of efforts over the past years in dealing with inexactness about \emph{informative} and \emph{computational} oracles in order for widely spread algorithms (Bregman Proximal Gradient, Prox-Linear, their accelerated variants, etc.) \cite{PD18,Stonyakin20,Yang21} to remain applicable in such scenarios. 
As argued in the previous paragraph, a common source of inexactness is due from the necessity of numerically solving auxiliary non-trivial optimization problems to produce oracles' outputs. 
\example \label{introductory_example} (\emph{saddle-point problems}) For instance,  \cite{Devolder2013FirstorderMW} analyzed the Gradient Method (GM) that allowed for what they defined as $(\delta,L,\mu)$ inexact oracles. Consider saddle-point problems of the type \begin{equation}
F^* = \min_{x\in\mathbb{R}^d}\,\Big\{F(x):= \max_{u \in \mathbb{R}^n}\,G(u) + \langle Au,\,x\rangle\Big\} > - \infty
\label{eq:example_intro}
\end{equation} where $G$ is $L(G)$ smooth and $\mu(G)>0$ strongly-concave and $A \in \mathbb{R}^{d \times n}$ is a matrix. They showed that $F$ is $L(F) = \frac{\lambda_{\text{max}}(AA^T)}{\mu(G)}$ smooth, $\mu(F) = \frac{\lambda_{\text{min}}(AA^T)}{L(G)}$ strongly-convex and its gradient is given as $\nabla F(x) = Au^*(x)$ where $u^*(x)$ is the exact minimizer of 
\begin{equation}
\max_{u \in \mathbb{R}^n}\,G(u) + \langle Au,\,x\rangle
\label{eq:primal}
\end{equation}
In the general case, \eqref{eq:primal} cannot be solved exactly. Nevertheless, it can be approximately solved up to $\delta>0$ global accuracy quite efficiently, e.g.\@ by dedicated accelerated first-order methods \cite{BTB22}. That is, instead of $u^*(x)$, one can provide $u_x \in \mathbb{R}^n$ such that \begin{equation} G(u^*(x))+\langle Au^*(x),x\rangle - \big(G(u_x)+\langle Au_x,\,x\rangle\big) \leq \delta \label{eq:primal_approx}
\end{equation}
and use approximate information in GM: $\tilde{F}(x) = G(u_x)+\langle Au_x,\,x\rangle - \delta$, $\nabla \edited{\tilde{F}}(x) = A u_x$. 
\vspace{1pt}

Thereby and luckily enough, it happens that one can \emph{tune} the quality of oracles by adequately choosing the amount of computational time spent on these problems. Looking back at the above  example, it is well-known that the number $\omega$ of iterations one must undertake to solve \eqref{eq:primal} \edited{up to $\delta$ accuracy \eqref{eq:primal_approx}} scales as $\kappa(G)^{-\frac{1}{2}}\,\log(\delta^{-1})$ where $\kappa(G) = \mu(G)/L(G)$.   After $N$ steps of GM involving a sequence of such inexact gradients with parameters $\{\delta_k\}_{k=0}^{N-1}$ \edited{at iteration $k = 0,\dots,N-1$} while using a constant stepsize $L^{-1}= \edited{(}2 L(F)\edited{)}^{-1}$, it is proven in \cite{Devolder2013FirstorderMW} that for 
$$ \edited{\hat{x}_N = \frac{\sum_{k=0}^{N-1}\,(1-\kappa)^{N-1-k}\,x_k}{\sum_{k=0}^{N-1}\,(1-\kappa)^{N-1-k}}}$$
the following guarantees held 
\begin{equation}
F(\hat{x}_N) - F^* \leq \frac{M(N) + 3\,\sum_{k=0}^{N-1}\,(1-\kappa)^{N-1-k}\,\delta_k}{\sum_{k=0}^{N-1}\,(1-\kappa)^{N-1-k}}\label{eq:init_devolder}
\end{equation}
with $\kappa = 4^{-1}\kappa(G)$, $M(N) = 2^{-1}\,LR^2\,(1-\kappa)^N$ and $R$ initial distance to a minimizer of \eqref{eq:example_intro}.
One observes the additive impact of such inexactness on GM's convergence. \vspace{5pt}
\\
It is worth noticing that inexactness also naturally occurs in the context of stochastic gradient methods. Within such framework one can \emph{tune} stochastic gradient and/or \edited{H}essian's bias by averaging more or less sample gradients and/or \edited{H}essians \cite{Kavis2019UniXGradAU,Kohler2017SubsampledCR}. \\

Despite the fact that convergence guarantees under a different level of inexactness at each iteration are well-established in the literature, very few works devise specific inexactness schedules as in \cite{iALM21,Doikov2020InexactTM}. This can provably constitute a miss of opportunity. Akin to our introductory example, let us assume that we dispose of relations linking the quality of oracles ($\delta$) with the computational efforts invested in the creation of their output ($\omega$). It is possible to retrieve optimal inexactness schedules taking into account the \emph{trade-off} between the computational price and the harm in terms of convergence guarantees, e.g.\@ \eqref{eq:init_devolder}, of the prescribed oracle precision. We detail our optimality criteria in Section \ref{sec:opt_criteria}. 
Informally, we aim in this work at answering the question:
\begin{center}
    \textit{How can we make the most of a computational budget \\ \edited{when using optimization methods dealing with oracle inexactness?}}
\end{center}
\vspace{-15pt}
\subsection{Related work}

At first sight and as is often the case in mathematical optimization research, the goal of our work can be summed up simply as an improvement of a worst-case convergence bound. Indeed, under the models we motivate in Section \ref{sec:TOM} and the assumptions we introduce in Section \ref{sec:opt_criteria}, we solve a specific instance of \emph{non-linear allocation} problems to provide our enhanced inexactness schedules in terms of convergence guarantees per computational cost unit. \\ \\
How straightforward might it sound, to the best of our knowledge there actually exists not so many previous works on that subject, i.e.\@ \emph{trade-off} optimality for iterative algorithms dealing with controllable inexactness. We contrast however this last sentence by reassuring that a bunch of new papers propose criteria involving parameter fixed \emph{relative} inexactness (see \cite{Kong21,Potzl22} and references therein). In opposition with the framework of \emph{absolute} inexactness, computational costs to achieve \emph{relative} inexactness are by nature prone to uncertainty and worst-case guarantees dependencies on this parameter often appears quite opaque. One usually leaves the relative inexactness parameter as fixed to a conservatively low wished terminal accuracy. Authors of \cite{BTB22} studied a mix between \emph{relative} and \emph{absolute} inexactness, highlighting the positive impact of incorporating \emph{absolute} oracle accuracy.\\ \\ With a similar approach to ours, \cite{Machart2012OptimalCT} focuses on (Accelerated) Proximal Gradient algorithms as analyzed by \cite{Schmidt2011ConvergenceRO} in their seminal paper. Unlike us, they do not show explicit closed-forms, \edited{i.e. analytical value for $\delta_k$, method's level of inexactness at iteration $k =0,\dots,N-1$ as in \eqref{eq:init_devolder}}, and their results remain mainly theoretical. \\
More recently, \cite{iALM21} comes up with specific schedules for the inexactness in the auxiliary subproblems they deal with in order to lower the overall computational load after $N \in \mathbb{N}$ iterations. Our more general results encompass theirs if we were concerned with the same context of inexact Augmented Lagrangian. As a byproduct of an asymptotical analysis, i.e.\@ $N \to \infty$, \cite{Villa13} whose follow-up resides in \cite{BTB22}, suggests to use inexactness schedules that decrease sufficiently fast in order to maintain (up to a logarithmic factor) the rate of convergence of the \emph{error-free}, i.e.\@ \emph{exact}, counterpart of the algorithms at scope. In \cite{Doikov2020InexactTM} were tried \emph{online} (adaptive) as well as purely \emph{offline} schedules in what concerns the accuracy of the inexact higher-order tensor steps. Both last works provide \emph{non-constant} inexactness schedules but none of them brings into balance (inexact) oracles' computational complexities.
\subsection{Contributions} 
Let us consider an iterative algorithm involving controllable inexactness, e.g\@ GM as depicted in the previous paragraph. We can state our main contributions as follows.
\begin{itemize}
\item
Firstly, considering fixed the number $N\in \mathbb{N}$ of inexact oracles calls, we propose a systematic \emph{offline} procedure to devise the amount \edited{$\delta_k$} of inexactness to adopt at each iteration $k \in \{0,\dots,N-1\}$ based on an oracle cost model and algorithm's guarantees. To that end, we solve a rather general \emph{non-linear assignment} problem, result of independent interest. Building upon this, we present closed-forms results about the optimal inexactness schedules for a broad class of oracle cost models, \edited{directly inspired from practical scenarios.}  \\
\item
Secondly, we \edited{propose an \emph{online} heuristic extension} in which neither $N$ nor the overall computational budget allocated must be fixed beforehand. \\
\item \edited{Thirdly,} we conduct numerical experiments that sustain the validity of our approach, \edited{either \emph{offline} or \emph{online}}, by comparing it against (a) \emph{constant} inexactness schedules and (b) \emph{non-constant} inexactness schedules from the literature \cite{Villa13,BTB22}.
We emphasize on the fact that our strategy is fully implementable.
\end{itemize}
\subsection{Outline}
At the end of the present section, we clarify our notations and define the useful concept of \emph{descending rank}. In Section \ref{sec:TOM} are motivated and then explained the concept of Tunable (Inexact) Oracles, i.e.\@  we develop the models of costful (inexact) oracles and convergence under an inexactness schedule which we illustrate thanks to three examples serving as guidelines. We then proceed to Section \ref{sec:opt_criteria} in which we show all the contributions teased just above. This section represents the main content of this paper. We conclude by quantifying the computational savings of our approach on the aforementioned guideline examples in Section \ref{sec:numericals}.
\subsection{Preliminaries}
We introduce some handy notations and a definition extensively used in this paper. 

\paragraph*{Sets} $\mathbb{N} = \{1,\dots,\infty\}$ will refer to the set of strictly positive integers. $\mathbb{R}_{+}$ (\edited{respectively} $\mathbb{R}_{-}$) contains all the non-negative (\edited{respectively} non-positive) real numbers. We denote $\mathbb{R}_{++} = \mathbb{R}_{+} \backslash \{0\}$ and $\mathbb{R}_{--} = \mathbb{R}_{-} \backslash \{0\}$. Let $n \in \mathbb{N}$, we define 
$[n] = \{0,\dots,n-1\}$.

\paragraph*{Sequences} Our enumerating indices start at $0$. We write column vectors $v \in \mathbb{R}^n$ as $v = (v_0,\dots,v_{n-1})^T$.
Depending of the context, we will equivalently write $v \equiv \{v_k\}_{k=0}^{n-1}$.

\paragraph*{Shorthands} $\mathbf{e} \in \mathbb{R}^n$ stands as the vector of size $n$ full of $1$. For any $\mathcal{S} \subseteq [n]$, $\mathbf{e}_{\mathcal{S}}$ is defined as follows: $(\mathbf{e}_{\mathcal{S}})_k = 1$ if $k \in \mathcal{S}$ and $0$ otherwise.
On the other hand, $\mathbf{0} \in \mathbb{R}^n$ stands as the vector of size $n$ full of $0$, i.e.\@ $\mathbf{0} = \mathbf{e}_{\emptyset}$.
For any $v \in\mathbb{R}^n$, we define the subvector  $(v)_{\mathcal{S}}$ of size $|\mathcal{S}| \leq n$ whose values are taken from $v$
at indices in $\mathcal{S}$. 
 \paragraph*{Operations on vectors} Let $p \in \mathbb{R}$ and $u,v \in \mathbb{R}^n$. We proceed to entry-wise operations like 
$v^p = (v_0^p,\dots,v_{n-1}^p)^T$,
$u \odot v = (u_0\cdot v_0,\dots,u_{n-1}\cdot v_{n-1})^T$,
$u^T v = \sum_{k=0}^{n-1}\,u_k\cdot v_k$ and 
$|v| = (|v_0|,\dots,|v_{n-1}|)^T$
More generally, applying any operator $o : \mathbb{R} \to \mathbb{R}$ in an element-wise fashion on a vector $v$ is authorized: 
 $o(v) = (o(v_0),\dots,o(v_{n-1}))^T$.
\paragraph*{$p$-norms} For any $p \geq 1$ and any  $v \in \mathbb{R}^{n}$, the $p$-norm of $v$ reads 
$||v||_p = (\mathbf{e}^T |v|^p)^{\frac{1}{p}}$.\\
As usual, we extend the notation with $p = \infty$ and set 
$||v||_{\infty} \overset{\Delta}{=} \max\{|v_k|\,|\,k \in [n]\}$.\\
In what concerns matrices, the $p$-norm of any element $V \in \mathbb{R}^{n \times d}$ translates to $$||V||_p = \sup_{u\not = \mathbf{0}}\frac{||Vu||_p}{||u||_p}$$
\edited{Unless stated otherwise, we understand norms and distances as Euclidean ones ($p=2$) throughout this paper, i.e. $||v|| := ||v||_2$ and $||V|| := ||V||_2$.}
\paragraph*{Lipschitz continuity} $g : \mathbb{R}^n \to \mathbb{R}^d$ is $L$ Lipschitz (continuous) with respect to a $p$-norm if for any $y,x \in \text{dom}(g)$, 
$||g(y)-g(x)||_{p^*} \leq L\,||y-x||_p$ with $p^*=p\cdot(p-1)^{-1}$.
\paragraph*{Inequalities}
Let $\mathcal{L} : \Xi \subseteq \mathbb{R} \to \mathbb{R} \cup \{\infty\}$, $\mathcal{J} : \Xi \subseteq \mathbb{R}\to \mathbb{R} \cup \{\infty\}$ be two functions. \edited{Let $\text{dom}(\mathcal{L}) = \text{dom}(\mathcal{J})$,} we consider the following: 
\begin{align*}
\mathcal{L} \succeq \mathcal{J} &\Leftrightarrow \mathcal{L}(\sigma) \geq \mathcal{J}(\sigma)\edited{,} \hspace{5pt} \forall \sigma \in \Xi\\
\mathcal{L} \succ \mathcal{J} &\Leftrightarrow \mathcal{L}(\sigma) > \mathcal{J}(\sigma)\edited{,} \hspace{5pt}\forall \sigma \in \text{dom}(\mathcal{J}) 
\end{align*}
For any $u,v \in \mathbb{R}^n$, we also note:
\begin{align*}
u \succeq v \Leftrightarrow u_k\geq v_k\edited{,} &\hspace{5pt} \forall k \in [n]\\
u \succ v \Leftrightarrow u_k > v_k\edited{,}&\hspace{5pt} \forall k \in [n]
\end{align*}

\definition
\label{descending_rank}
(\emph{descending rank}) Let $\nu \in \mathbb{R}^n$ and let $\rho : [n] \to [n]$ be a bijection such that the vector $\hat{\nu} = (\nu_{\rho^{(-1)}(0)},\dots,\nu_{\rho^{(-1)}(n-1)})^T$ is sorted in \emph{descending} mode.  \\For any $k \in [n]$, we call $\rho(k)$ the \emph{descending rank} of $\nu_k$ (with respect to $\rho$).

\remark It comes that if $\rho(k) = j$ then $\nu_k$ is, according to the sorting induced by $\rho$, the $(j+1)$-th biggest element of $\nu$. 

\section{Tunable (Inexact) Oracles}
\label{sec:TOM}
Here below we aim at defining the class of Tunable Oracles Methods (TOM) from which one can benefit by adapting the amount of inexactness in the oracles involved at each iteration. Prior to this goal, we recall some findings about methods incorporating inexactness and we introduce our assumed oracle cost model. Finally, we substantiate the concept of Tunable Oracles Methods with thee complete examples serving as common thread.
\subsection{Impact of Inexact Oracles}
The worst-case behaviour analysis of iterative methods relying on inexact oracles displays some favorable structure. Let $N \in \mathbb{N}$ be the number of performed iterations, one defines the total amount of inexactness at iteration $k \in \{0,\dots,N-1\}$ as $\delta_k$. Researchers come up with convergence models
\begin{equation}
\mathcal{C}(N) \leq \mathcal{E}(\{\delta_k\}_{k=0}^{N-1})
\label{eq:convergence_model}
\end{equation}
where
\begin{itemize}
    \item $\mathcal{C}(N)$ acts as a positive \emph{gauge} one aims to minimize
    \item $\mathcal{E} : \mathbb{R}^{N}_+ \to \mathbb{R}_+$ informs about convergence under the impact of a sequence $\{\delta_k\}_{k=0}^{N-1}$
\end{itemize}

\hspace{-10pt}$\mathcal{C}$ gauges include gradient mapping norms \cite{PD18}, functional gaps \cite{Stonyakin20,Devolder2013FirstorderMW}. \\As in Example \ref{introductory_example} inequality \eqref{eq:init_devolder}, a possible instance for our model could show up as  
$$ \underbrace{F(\hat{x}_N)-F^*}_{\mathcal{C}(N)} \leq \underbrace{\frac{M(N) + 3\,\sum_{k=0}^{N-1}\,(1-\kappa)^{N-1-k}\,\delta_k}{\sum_{k=0}^{N-1}\,(1-\kappa)^{N-1-k}}}_{\mathcal{E}(\{\delta_k\}_{k=0}^{N-1})}$$
\remark \label{inexact_model_comments2} Within the \emph{error-free} framework, i.e. $\delta_k = 0$ for every integer $k$, $$\lim_{N \to \infty}\, \mathcal{E}(\{0\}_{k=0}^{N-1}) = 0 $$
When accounting for errors, one theoretically observes convergence only up to $\mathcal{E}(\{\delta_k\}_{k=0}^{N-1})$ accuracy. It can happen that the schedule $\{\delta_k\}_{k=0}^{N-1}$ does not decrease sufficiently fast leading to a well-known phenomenon referred to as \emph{error-accumulation} in the literature \cite{Devolder13}. This latter translates to 
$$\lim_{N \to \infty}\,\mathcal{E}(\{\delta_k\}^{N-1}_{k=0}) = \infty$$ 
\subsection{Cost of Inexact Oracles}
Without loss of generality, one can define a \emph{reference} inexactness $\bar{\delta} > 0$ together with constants $0\leq m < 1 < M < \infty$ such that the inexactness of the oracles fall in the segment $\Xi := [m\,\bar{\delta},M\,\bar{\delta}]$. This means that at each iteration $k \in \{0,\dots,N-1\}$, one allows the user to freely pick up any $\delta_k \in \Xi$. As previously unveiled, such request costs a computational tribute, namely $\mathcal{B}_k(\delta_k) \geq 0$. 
Suitable for a variety of applications \cite{BTB22,PD18,Kavis2019UniXGradAU}, we suggest the oracle cost model 
\begin{equation}
\mathcal{B}_k(\delta_k) = b_k \,h(\delta_k)
\label{eq:cost_model}
\end{equation}
where 
\begin{itemize}
    \item $b_k > 0$ denotes the cost distortion of iteration $k$
    \item $h : \Xi \to \mathbb{R}_+$ dictates how the cost fluctuates with $\delta_k$
\end{itemize} 
\newpage
\hspace{-12pt}Recalling once again the introduction, one identifies the \emph{a priori} number $\omega_k$ of inner-iterations to obtain $\delta_k$ inexact information about $F$ as a multiple of $$ \mathcal{B}_k(\delta_k) = \underbrace{\sqrt{\kappa(G)^{-1}}}_{b_k}\,\underbrace{\log(\delta_k^{-1})}_{h(\delta_k)}$$ 

\remark \label{conventional} When unknown or when no argument justifies that any iteration turns out to be less expensive than another, we arbitrarily set $b_k = 1$ for any integer $k \geq 0$. Let us also point out that in the absence exogenous indication, $m=0$ and $M=\infty$. 
\vspace{-10pt}
\subsection{Tunable Oracle Method}
An instance of Tunable Oracle Methods (TOM) stands as a $N$ steps \emph{iterative} algorithm $\mathcal{A}$ whose iterations involve a notion of \emph{controllable} $\delta$-inexactness as in \eqref{eq:cost_model} such that $\mathcal{A}$ converges in the sense of \eqref{eq:convergence_model}. Furthermore, in order to truly exploit such \emph{controllable} feature, we require the explicit knowledge of $\{b_k\}_{k=0}^{N-1}$, $\mathcal{E}$ and $h$, up to a multiplicative constant factor. 
Such desire emphasizes the \emph{offline} nature focus of the paper at this stage. We make use of \emph{a priori} information about the behaviour of $\mathcal{A}$ and the expected cost one should encounter while requesting a schedule of $\{\delta_k\}_{k=0}^{N-1}$ inexact oracles. \\We explain later on how to take advantage of FOM in an \emph{online} setting.
\subsection{Guideline examples}
Throughout the sequel, we \edited{make use of} the following: 
\begin{itemize}
    \item $\Psi : \mathbb{R}^d \to \mathbb{R} \cup \{\infty\}$ proper closed convex function, not necessarily smooth.
    \item $\ell: \mathbb{R}^{q} \to \mathbb{R}$ closed convex function, $L_\ell < \infty$ Lipschitz continuous and simple,\\ i.e.\@ its proximal operator can be computed at negligible cost\footnote{Existence of a closed-form or at the expense of an easy one dimensional segment-search.}.
    \item $c : \mathbb{R}^d \to \mathbb{R}^{q}$ a smooth map with $L_{\nabla c} < \infty$ Lipschitz continuous Jacobian $\nabla c$.
\end{itemize} 
Combining these ingredients, we finally introduce $f = \ell \circ c$ and $F = f + \Psi$ nonconvex, nonsmooth in the general case. We will assume that $$F^* = \min_{x\in\mathbb{R}^d}\,F(x) > - \infty$$

\example
\label{iAFB} (\emph{composite convex optimization with inexact proximal operator}) \\
The inexact Accelerated Forward-Backward algorithm (iAFB) from
\cite{BTB22} tackles so called convex composite problems, ubiquitous in image processing \cite{CP16}. The specificities of this class read: $q=1$, $\ell = \text{id}_{\mathbb{R}}$, convexity (\edited{respectively} $\mu\geq0$ strong-convexity) of $c$ (\edited{respectively} $\Psi$) thus $f = c$ and $\nabla f$ is $L_f=L_{\nabla c}$ Lipschitz continuous. Among other appealing features, iAFB presented therein allows for $\Psi$ whose proximal mapping is not simple, e.g.\@ in \emph{sparse overlapping groups regularization} \cite{Jen10}. At $z\in \mathbb{R}^d$, $\lambda > 0$, the primal-dual pair of problems related to the proximal step of a function $\phi$ translate to 
\begin{align*}
    \min_{x \in \mathbb{R}^d}&\,\bigg\{\Phi_p(x;\lambda,z,\phi) := \frac{1}{2}\,||x-z||^2 + \lambda\,\phi(x)\bigg\} & \hspace{3pt} (P)\\
    \max_{v \in \mathbb{R}^d}&\,\bigg\{\Phi_d(v;\lambda,z,\phi) := \frac{1}{2}\,(||z||^2-||z-\lambda v||^2) - \lambda\,\phi^*(v)\bigg \} & \hspace{3pt} (D)
\end{align*}
where $\phi^*$ is the Fenchel conjugate of $\phi$.
iAFB produces iterates $x_k$, $y_k$, $z_k$, $v_k$ \edited{thanks} to coefficients $\lambda_k \in \mathcal{O}(L_{f})$ and $A_k \in \Omega(k^2)$ for any integer $k\geq 0$. \edited{Either iAFB employs predefined sequence of stepsizes $\{\lambda_k^{-1}\}_{k \in \mathbb{N}}$ and then $A_k$, accounting as certificate sequence, can be computed in advance, i.e. \emph{offline}, or it adopts Armijo line-search to adapt to local smoothness $\lambda_k$ and one only has access to $A_k$ for iterations $k^{'} \geq k$, in an \emph{online} fashion.}
Following authors' notations, $(x_{k+1},v_{k+1})$ stems as a $\delta_k$ inexact output from the proximal step oracle at iteration $k \geq 0$ if \begin{equation}
\text{PD}_{\lambda \phi}(x,v;z) := \Phi_p(x;\lambda,z,\phi)-\Phi_d(v;\lambda,z,\phi) \leq \delta_k \label{eq:first_requirement} \end{equation} 
with $x = x_{k+1}$, $v = v_{k+1}-\mu \,x_{k+1}$, $\lambda = \frac{\lambda_k}{1+\lambda_k \mu}$, $\phi = \Psi-\frac{\mu}{2}||\cdot||^2$ and $z = \frac{y_k-\lambda_k \nabla f(y_k)}{1+\lambda_k \mu}$. \vspace{5pt}

\hspace{-12pt}Under this notion of inexactness, after $N \in \mathbb{N}$ steps, the following guarantees hold \begin{equation}
\underbrace{F(x_N) - F^*}_{\mathcal{C}(N)} \leq \underbrace{\frac{4R^2 + \sum_{k=0}^{N-1}\,(A_{k+1}\cdot(1+\mu\lambda_k)^2\cdot\lambda_k^{-1})\,\delta_k }{A_N}}_{\mathcal{E}\big(\{\delta_k\}_{k=0}^{N-1}\big)}
\label{eq:convergence_model_example_sub}
\end{equation}
where $R < \infty$ denotes the distance from $x_0$ to the set of global minimizers of $F$.\\ 
The presumed oracle cost, i.e.\@ the work $\mathcal{B}_k(\delta_k)$ needed to produce $(x_{k+1},v_{k+1})$ fulfilling \eqref{eq:first_requirement} remains problem dependent. It is also highly influenced by the technique employed to solve the pair (P) / (D). In \cite{BTB22}, authors deal with \texttt{CUR} factorization problem with \emph{sparse overlapping groups regularization}. They simply use FISTA, a \emph{fast} first-order method, to solve (D) yielding a linked sequence of primal recovered iterates (P). In the present setting, it is shown in \cite{Lu16} that the primal-dual gap decreases sublinearly as $\omega_k^{-1}$ where $\omega_k$ denotes the number of steps of FISTA in (D). Then, for an arbitrarily chosen $b_k = 1$ (see Remark \ref{conventional}), we model the cost as
\begin{equation} 
\mathcal{B}_k(\delta_k) = \underbrace{\delta_k^{-1}}_{h(\delta_k)}
\label{eq:cost_model_example_sub}
\end{equation}

\example
\label{example_detailed} (\emph{composition of convex functions optimization}) The paper \cite{PD18} considers the most general setting for which global optimality is, obviously, out of reach. Let us focus on a single method analyzed therein, the inexact Prox-Linear algorithm, iPL in short \edited{terms}. Given $z \in \text{dom}(\Psi)$ and $t>0$, one defines the functional $$F_{t}(\cdot;z) = \Psi(\cdot) + \ell(c(z)+\nabla c(z)(\cdot-z)) + \frac{t^{-1}}{2}||\cdot-z||^2$$ We call $\delta$ inexact solution $(x_+,\xi)$ of the minimization of $F_{t}(\cdot\,;\,z)$ a pair fulfilling the properties: $||\xi||\leq \delta$ and $$x_{+} \in \arg \min_{x\in\mathbb{R}^d}\Psi(\cdot) + \ell(\xi + c(z)+\nabla c(z)(\cdot-z)) + \frac{t^{-1}}{2}||\cdot-z||^2$$ In the case of an exact $(x_+,0)$ (i.e.\@ $\delta=0$ inexact) solution, the proximal gradient mapping $\mathcal{G}_{t}(x)$ at $z$ with stepsize \edited{$t\leq(L_\ell\,L_{\nabla c})^{-1}$} stands as the vector $ \mathcal{G}_{t}(x) = t^{-1}(x-x_+)$ if and only if $F(x_+) \leq F_t(x_+;z)$. Its norm constitutes a measure of stationarity \cite{PD18} that generalizes the gradient norm in the smooth unconstrained optimization framework. When calling up $\{\delta_k\}_{k=0}^{N-1}$ inexact solutions at successive iterates $z=x_k \in \text{dom}(\Psi)$ and stepsizes \edited{$t = t_k \in ]0,(L_\ell\,L_{\nabla c})^{-1}]$}, iPL produces iterates $\{x_{k}\}_{k\geq0}$ such that: 
\begin{equation}
\underbrace{\min_{k^{'} \in \{1,\dots,N\}}\,||\mathcal{G}_{t_k}(x_{k^{'}})||^2}_{\mathcal{C}(N)} \leq \underbrace{\frac{\sum_{k=0}^{N-1}4\,t_k^{-1}\Big(F(x_k)-F(x_{k+1}) + 4\, L_\ell\,\delta_k\Big)}{N}}_{\mathcal{E}\big(\{\delta_k\}_{k=0}^{N-1}\big)}
\label{eq:convergence_model_example}
\end{equation}
Again, the oracle cost should be derived from the complexity of the inner method used to obtain the sequence of $\{\delta_k\}_{k=0}^{N-1}$ inexact solutions at pairs $\{(x_k,t_k)\}_{k=0}^{N-1}$. When no further specific structure is taken as granted, excepted an easy proximal mapping of $\Psi$ (unlike the previous example \eqref{eq:convergence_model_example_sub}), \cite{PD18} suggests to exploit duality. At any iteration $k \geq 0$, one can compute a $\delta_k$ inexact pair $(x_{k+1},\xi_{k+1})$ as previously explained by finding a subgradient $\xi_{k+1}$ of the (negated) Fenchel conjugate function of $F_{t_k}(\cdot\,;\,x_{k})$ whose norm does not exceed $\delta_k$. With the oracles at hand and accounting for the fact that this conjugate boils down to a sum of a smooth convex term and a proximable nonsmooth convex function, A-HPE from \cite{Monteiro13} ensures a minimal subgradient norm of order  $\mathcal{O}(t_k ||\nabla c(x_k)||^2 \,\omega_k^{-\frac{3}{2}})$ after $\omega_k$ A-HPE iterations.\\  \\ 
\edited{Hence, it comes that the number of A-HPE iterations required to get a $\delta_k$ inexact pair scales as $t_k^{\frac{2}{3}}\, ||\nabla c(x_k)||^{\frac{4}{3}}\,\delta_k^{-\frac{2}{3}}$. Obviously, one does not know in advance the value $||\nabla c(x_k)||$ since it relies on $x_k$, only computable from iteration $k \geq 0$. Remark \ref{conventional} then suggests that any $\gamma >0$, e.g. $\gamma=1$, can serve as artificial upper-bound so that either iPL uses predefined stepsizes $\{t_k\}_{k\in\mathbb{N}}$ (\emph{offline}) or adjusted stepsizes (\emph{online}).}
\begin{equation} 
\mathcal{B}_k(\delta_k) = \underbrace{t_k^{\frac{2}{3}}}_{b_k}\,\underbrace{\delta_k^{-\frac{2}{3}}}_{h(\delta_k)}
\label{eq:cost_model_example}
\end{equation}
Note that if $\text{dom}(\Psi)$ is bounded with known diameter, one can use regularization and apply Nesterov fast gradient method FGM on strongly-convex composite objectives in order to obtain an enhanced  $t_k^{\frac{1}{2}}\,||\nabla c(x_k)||\,\delta_k^{-\frac{1}{2}}$ complexity \cite{Yu13}. We emphasize that \emph{stopping criteria} are readily available to check whether $\delta_k$ accuracy has been reached. \\
\vspace{-5pt}
\example
\label{example_detailed_2} \edited{(\emph{robust optimization on convex hull}) Closely related to \cite{BTB22}, \cite{Stonyakin20} extends the analysis of Fast Gradient Method (FGM) under the presence of inexactness by allowing $f$ to be smooth and convex relative to some Legendre kernel function (see \cite{Lu16}), covering the so-called Bregman setting. Therefore, their results also encompass the Euclidean setting which we stick to for the sake of simplicity. We assume that $\Psi$ represents an indicator function of a convex subset $X \subseteq \mathbb{R}^d$, $q=1$, $\ell = \text{id}_{\mathbb{R}}$ and $c$ is $\mu \geq 0$ strongly-convex. That is, $f$ is $\mu \geq 0$ strongly-convex and $\nabla f = \nabla c$ is $L_f = L_{\nabla c}$ Lipschitz continuous. Let $\Theta = \{\theta_i\}_{i=1}^{n}$ be a collection of $n \in \mathbb{N}$ vectors from $\mathbb{R}^d$ dubbed as \emph{scenarios} and let $\sigma > 0$. In \emph{robust optimization}, one might be interested in minimizing the function 
$$f : X \to \mathbb{R},\hspace{2pt} x \to f(x) = \frac{\mu}{2}\,||x||^2 + \max_{\theta \in \textbf{conv}(\Theta)}\, \langle\theta,x\rangle - \frac{\sigma}{2}||\theta- \bar{\theta}||^2$$
for some \emph{anchor scenario} $\bar{\theta}$, e.g. $\bar{\theta} = n^{-1}\,\sum_{i=1}^{n}\,\theta_i$. \newpage In other words, one would like to minimize a (regularized) linear objective taking into account that the cost vector could be any convex combination of previously encountered costs $\{\theta_i\}_{i=1}^{n}$. Akin to Example \ref{introductory_example}, we deduce that the \emph{exact} gradient of $f$ at any $x \in X$ reads $\nabla f(x) = \mu x + \theta^*(x)$ where $\theta^*(x) = \arg \max_{\theta \in \textbf{conv}(\Theta)}\, \langle \theta,x\rangle - \frac{\sigma}{2}||\theta- \bar{\theta}||^2$ and gradient's Lipschitz constant $L_{f} = \sigma^{-1}$.} \\ \vspace{3pt} 

\hspace{-10pt}\edited{An approximate maximizer $\theta_x \in \textbf{conv}(\Theta)$ of the problem defining $f$ at $x$ that verifies
\begin{equation}
\langle(\theta^*(x) - \theta_x), x\rangle + \frac{\sigma}{2}\big(||\theta_x-\bar{\theta}||^2-||\theta^*(x)-\bar{\theta}||^2\big) \leq \delta
\label{eq:approx_inner}
\end{equation}
can be used to construct $\tilde{f}(x) = \frac{\mu}{2}||x||^2 + \langle\theta_x, x\rangle - \frac{\sigma}{2}||\theta_x-\bar{\theta}||^2 \simeq f(x)$ and $\nabla \tilde{f}(x) = \mu 
x+\theta_x\simeq \nabla f(x)$, providing $ (2\delta,2L_f+\mu,\mu)$ inexact information as originally understood in \cite{Devolder13}.}
\edited{\cite{Stonyakin20} show that FGM (\cite{Stonyakin20}, Algorithm 2) involving a sequence of $\{\delta_k\}_{k=0}^{N-1}$ inexact information and fixed stepsizes produces a final iterate $x_N$ such that \begin{equation}
\underbrace{F(x_N) - F^*}_{\mathcal{C}(N)} \leq \underbrace{\frac{R^2 + 2\,\sum_{k=0}^{N-1}\,A_{k+1}\,\delta_k\,}{A_N}}_{\mathcal{E}\big(\{\delta_k\}_{k=0}^{N-1}\big)}
\label{eq:convergence_model_example_2}
\end{equation}
where, once again, $R < \infty$ denotes the distance from $x_0$ to the set of global minimizers of $F$. As for iAFB, the coefficients $A_k$ for $k\geq0$ serve as convergence certificates and are involved in subtle convex combinations of iterates within FGM. In the present setting, $A_k$'s value can be determined beforehand, $A_k\in{O}(\max\{k^2,(1+\frac{1}{4}\sqrt{\mu /(\mu+2 L_{f})})^{2k}\})$.}\\
\edited{Let $O = [\theta_1,\dots,\theta_n]^T$ and let $\hat{\kappa} := \frac{\lambda_{\text{min}}(OO^T)}{\lambda_{\text{max}}(OO^T)}$. 
By using another version of FISTA described in \cite{vdb22}, one can take advantage of possible strong-convexity, i.e. $\hat{\kappa}>0$, of the usual reformulation of the inner-problem : 
$$ \max_{\theta \in \textbf{conv}(\Theta)}\, \langle\theta,x\rangle - \frac{\sigma}{2}||\theta- \bar{\theta}||^2 = \max_{w\succeq \mathbf{0},\, ||w||_1=1}\,\langle O^Tw ,x\rangle-\frac{\sigma}{2}||O^Tw -\bar{\theta}||^2 $$Assuming that $\hat{\kappa}>0$, just as in Example \ref{introductory_example}, one can link the work $\omega_k$ to obtain $\delta_k$ accurate $\theta_{x_k}$ at iteration $k \geq 0$ by writing
\begin{equation} 
\mathcal{B}_k(\delta_k) =  \underbrace{\sqrt{\hat{\kappa}^{-1}}}_{b_k}\,\underbrace{\log(\delta_k^{-1})}_{h(\delta_k)}
\label{eq:cost_model_example_2} 
\end{equation}
whereas in the absence of strong-convexity, i.e. $\hat{\kappa}=0$, one would rather set $\mathcal{B}_k(\delta_k) = \sqrt{\lambda_{\text{max}}(OO^T)} \,\delta_k^{-\frac{1}{2}}$ as common for smooth convex optimization \cite{Yu13}.
We explain in Section \ref{sec:numericals} how one can easily monitor the quality of an approximate solution for \eqref{eq:approx_inner}.}
\remark
\label{continuity_work} Usually, the workload $\omega \propto \mathcal{B}(\delta)$ to obtain (the output) of $\delta$ inexact oracles stays \emph{by nature} an integer quantity, e.g.\@ a number of inner-iterations (Example \ref{iAFB}, \ref{example_detailed} and \ref{example_detailed_2}).  For the sake of simplicity however, we will accept that it varies continuously as in the formulas \eqref{eq:cost_model_example_sub}, \eqref{eq:cost_model_example}, \eqref{eq:cost_model_example_2}.

\section{Optimal Inexactness Schedules}
\label{sec:opt_criteria}

Now that the impact and the cost of inexact oracles have been introduced, we can elaborate about the main objective of this paper. We assume for the time being that $N \in \mathbb{N}$ is fixed. One can try to suffer the smallest possible effect from inexact oracles in order to ensure the best worst-case convergence upper-bound. Obviously, one would like each $\delta_k$ to match its best value $m \bar{\delta}$. However, sometimes one cannot afford a schedule $\{\delta_k = m \bar{\delta}\}_{k=0}^{N-1}$ if one limits the overall computational budget. On the other hand, when not obliged to, it is not advisable to ask for the worst oracle accuracy at each iteration, i.e. $M\bar{\delta}$, yielding the inexactness schedule $\{\delta_k = M \bar{\delta}\}_{k=0}^{N-1}$. \\ \\
Thereby, we propose to solve a master problem that aims at devising the optimal \emph{trade-off} between the costs of oracles and worst-case guarantee harms due to their inexactness. We start by providing our blanket Assumptions A and B.
\subsection{Framework}

From now on we will make a little abuse of notation about the \emph{reference} inexactness $\bar{\delta} > 0$. Depending on the context, this latter will either depict a real value, a $N$ steps inexactness schedule $\{\delta_k = \bar{\delta}\}_{k=0}^{N-1}$ or even the corresponding vector in $\mathbb{R}^{N}_+$, i.e. $\bar{\delta}\,\mathbf{e}$.
\paragraph*{Assumption A.}
\label{assumption_inexact_model}
 We have access to $a \succ \mathbf{0}$, $\exists\,\tilde{\mathcal{E}} : \mathbb{R}_+ \to \mathbb{R}_+$ \emph{strictly increasing} with
\begin{equation}
\mathcal{E}(\delta):=
\mathcal{E}(\{\delta_k\}_{k=0}^{N-1}) = \tilde{\mathcal{E}}\Bigg(\sum_{k=0}^{N-1}\,a_k\,\delta_k\Bigg) 
\label{eq:mathcal_e_model}
\end{equation}

We now invoke Assumption B suggesting that one is able to predict the overall cost $\mathcal{B}(\delta)$ of a schedule of inexactness $\delta$. Furthermore, some technicalities about the structure of the function $h$ from \eqref{eq:cost_model} are stated.
\vspace{-10pt}
\paragraph*{Assumption B.} \label{assumption_cost_model} We have access to $b \succ \mathbf{0}$ and $h : \Xi \to \mathbb{R}_+$ \emph{differentiable},  \emph{invertible}, \emph{strictly decreasing} with
\begin{equation}
\mathcal{B}(\delta):= \sum_{k=0}^{N-1}\,\mathcal{B}_k(\delta_k)  = \sum_{k=0}^{N-1}\,b_k\,h(\delta_k)
\label{eq:mathcal_h_model}
\end{equation}
$h' : \text{dom}(h') \supseteq \text{int}(\Xi) \to \text{Im}(h')\subseteq \mathbb{R}_{-}$ must be \emph{invertible} and \emph{strictly increasing}. 

\remark
\label{equivalence} Assumption A tells us that on any subset of $\mathbb{R}^N_+$, minimizing $\mathcal{E}(\{\delta_k\}_{k=0}^{N-1})$  boils down to minimizing $\sum_{k=0}^{N-1}\,a_k\,\delta_k$, i.e.\@ a minimizer of  $a^T\,\delta$ stays optimal for $\mathcal{E}(\delta)$. 

\renewcommand{\arraystretch}{1.6}
\remark
\label{illustration_example} \edited{In what concerns our examples, one can identify problem dependent constants $C_1,C_2 > 0$ such that $\tilde{\mathcal{E}}$ admits a shared structure $s \to \tilde{\mathcal{E}}(s) = C_1+C_2\,s$.}  \newpage
\begin{table}[h!]
\centering
\caption{Illustration of Assumptions A and B.}
\begin{tabular}{c|c|c|c}
 & $h(\delta_k)$ & $\propto a_k$ &$\propto b_k$\\
\hline
\hline
  Example \ref{iAFB} & $\delta_k^{-1}$ & $A_{k+1}\cdot \edited{(1+\mu \lambda_k)^2} \cdot \lambda_k^{-1}$ & $1$ \\
  \hline
   Example \ref{example_detailed} &  $\delta_k^{-\frac{2}{3}}$& $t_{k}^{-1}$& $t_k^{\frac{2}{3}}$\\
   \hline
   Example \ref{example_detailed_2} & $\log(\delta_k^{-1})$ / $\delta_k^{-\frac{1}{2}}$ & $A_{k+1}$ &$1$\\ 
\end{tabular}
\label{tab:my_tabular_AB}
\end{table}

In their adaptive \edited{\emph{online}} versions, sometimes practically more attractive, parameters $\lambda_k$ (iAFB) and $t_k$ (iPL) are not known in advance since they rely on line-searches at iteration $k\geq0$. Unfortunately, they influence the values of $a_k/b_k$'s. Thenceforth, Example \ref{iAFB} and \ref{example_detailed} would fail to satisfy Assumption A and/or B. Nevertheless, as already argued, one can fix $\lambda_k \leq L_{\nabla c}^{-1}$ (iAFB), $t_k^{-1} \geq L_\ell\cdot L_{\nabla c}$ (iPL) for every $k \geq 0$ and avoid the line-searches. In such \edited{\emph{offline}} circumstances, both $\{a_k\}_{k\geq0}$ and $\{b_k\}_{k\geq0}$ become accessible and the assumptions are fulfilled. If taken constants, i.e. for every $k\geq 0$, $\lambda_k = \lambda$, $t_k=t$ for well-chosen $\lambda, t \in \mathbb{R}_+$, then $A_k \in \mathcal{O}(\max\{k^2,(1-\sqrt{\mu / L_{f}})^{-k}\})$ for iAFB, $A_k \in \mathcal{O}(\max\{k^2,(1+\frac{1}{4}\sqrt{\mu /(\mu+2 L_{f})})^{2k}\})$ in FGD and $t_k \in \mathcal{O}(1)$ in iPL.

\subsection{Master Problems}
\label{optimal_schedules}
\paragraph*{Accuracy controlled}
In its most standard version, we design a master problem \edited{$\delta^*(N,\bar{\delta},m,M)$} for which the degrees of freedom reside in the accuracies of the $N$ iterations of our TOM, i.e\@ $\delta \in \mathbb{R}_+^N$. Given a \emph{reference} $\bar{\delta} \in \mathbb{R}_+$ precision, we implicitly deduce the total allocated computational budget as $\sum_{k=0}^{N-1}\,\mathcal{B}_k(\bar{\delta}) = \big(\sum_{k=0}^{N-1}\,b_k\big)\,h(\bar{\delta})$.\\Controlling $\delta \in [m \bar{\delta},\,M \bar{\delta}]^N$, we minimize the bound $\mathcal{E}(\delta)$ under the budget constraint $\mathcal{B}(\delta) = \mathcal{B}(\bar{\delta})$. Taking into account Assumptions A, B and Remark \ref{equivalence}, one translates
\begin{equation}
\edited{\delta^*(N,\bar{\delta},m,M)} \in \arg \min_{m\, \bar{\delta}
\, \preceq \, \delta \,\preceq\, M\,\bar{\delta}}
\,\sum_{k=0}^{N-1}\,a_k\,\delta_k \hspace{5pt}\text{s.t.}\hspace{5pt} \sum_{k=0}^{N-1}\,b_k\,h(\delta_k) = \mathcal{B}(\bar{\delta})
\label{eq:master_problem}
\end{equation}

\subsubsection{General solutions}
 We present a first theorem which will proved to be useful to guarantee some consistency in the optimal schedules.
Its scope encompasses more general structured \emph{non-linear} resource allocation problems. Related results can be found deeply rooted in the literature, the reader should refer to \cite{Mura98} and the references therein.

\begin{theorem} \label{consistency} Let $N \in \mathbb{N}$, $\Xi = [l,u] \subseteq \mathbb{R}$ and $D > 0$. Let $\{n_k\}_{k=0}^{N-1}$ and $\{d_k\}_{k=0}^{N-1}$ be two sequences of continuously differentiable functions \edited{on $\Xi$} \edited{and $\{n_k\}_{k=0}^{N-1}$ elements are furthermore strictly} increasing on $\Xi$. If for every $k \in [N]$, the ratio $\mathcal{I}_k = n'_k / d'_k$
is strictly \edited{increasing} and \edited{strictly} negative on $\Xi$ then any solution $\sigma^*$ of \begin{equation} \min_{\sigma \in \Xi^N}\,\sum_{k=0}^{N-1}\,n_k(\sigma_k) \hspace{5pt}\text{s.t.}\hspace{5pt} \sum_{k=0}^{N-1}\,d_k(\sigma_k) = D
\label{eq:abstract_general}
\end{equation}
admits as a property that for all pairs of indices $(k_1,k_2) \in [N]^2$, 
\begin{equation}
\mathcal{I}_{k_{1}} \succeq  \mathcal{I}_{k_{2}} \Rightarrow \sigma^*_{k_{1}} \geq \sigma^*_{k_{2}}
\label{eq:consistency}
\end{equation}
Moreover, 
\begin{equation}
\mathcal{I}_{k_{1}} \succ \, \mathcal{I}_{k_{2}} \hspace{3pt}\wedge \hspace{3pt} \sigma_{k_{1}}^*,\sigma_{k_{2}}^* \,\in\, ]l,u[ \,\Rightarrow \sigma^*_{k_{1}} > \sigma^*_{k_{2}}
\label{eq:consistency2}
\end{equation}

\end{theorem}

\begin{proof}
    We prove Theorem \ref{consistency} in Appendix \ref{proof_theorem1}.
\end{proof}

\example \label{general_alloc_refined} \edited{The reader can get more insight thanks to the following example. \\For every $k\in[N]$, let $\sigma \to n_k(\sigma) = (k+1)^3\,\sigma$, $\sigma \to n_k(\sigma) = -\log(\sigma)$, $l=1$ and $u = 10^5$. The functional ratios $\mathcal{I}_k$ are given by $$ \mathcal{I}_k \,:\, ]l,u[ \to \mathbb{R},\hspace{2pt} \sigma \to \mathcal{I}_k(\sigma) = \frac{n_k^{'}(\sigma)}{d_k^{'}(\sigma)}=\frac{(k+1)^3}{-\sigma^{-1}} = -(k+1)^3\,\sigma^{-1}$$ On $\Xi = [l,u]$, $n_k$'s and $d_k$'s are continuously differentiable and $n_k$'s are increasing while $\mathcal{I}_k$'s are strictly increasing and strictly negative. In addition, if $k_1 \leq k_2 $,
$$ \mathcal{I}_{k_{1}} \succeq \mathcal{I}_{k_2}$$ since for every $\sigma \in \Xi$, $\mathcal{I}_{k_{1}}(\sigma) = -(k_1+1)^3\,\sigma^{-1} \geq -(k_2+1)^3\,\sigma^{-1} = \mathcal{I}_{k_2}(\sigma)$.}

\paragraph*{Reordering} In the case where all the ratios $\{\mathcal{I}_{k}\}_{k=0}^{N-1}$ can be ordered, i.e. there exists a bijective mapping $\tau : [N] \to [N] $ such that for all pair of indices $(k_1,k_2) \in [N]^2$, 
\begin{equation}
    \tau(k_1) < \tau(k_2) \Rightarrow \mathcal{I}_{k_{1}} \succeq \mathcal{I}_{k_{2}}
    \label{eq:ordering_general}
\end{equation}
Theorem \ref{consistency} suggests that the entries $\{\sigma_k^*\}_{k=0}^{N-1}$ of an optimal solution $\sigma^*$ of \eqref{eq:abstract_general} can be \emph{ranked} as well. Provided a suitable comparison vector $\nu \in \mathbb{R}^N$, we advocate the usefulness of $\rho$, \emph{descending rank} function based on $\nu$ to act like a $\tau$ function above, i.e.\@ for any $k \in [N]$, we would have 
\begin{equation}
\rho(k) = \tau(k)
\label{eq:inphase}
\end{equation}
As a reminder from Definition \ref{descending_rank}, we write $\rho(k) = j$ if $\nu_k$ is the ($j+1$)-th largest element of $\nu$. It follows from Theorem \ref{consistency} that $\sigma_k^*$ must correspond to the ($j+1$)-th largest element of $\sigma^*$. 
As displayed in Theorem \ref{general_theorem}, Theorem \ref{consistency} applies verbatim to problem \eqref{eq:master_problem} by picking for every $k \in [N]$, $n_k(\sigma_k)=a_k\,\sigma_k$, $d_k(\sigma_k)=b_k h(\sigma_k)$ and $D=\mathcal{B}(\bar{\delta})$. One valid comparison vector that fulfills \eqref{eq:ordering_general} and \eqref{eq:inphase} would be $\nu_k = b_k / a_k$.
\vspace{2pt}

\hspace{-11pt}\edited{Let $\delta^*$ be an optimal solution for \eqref{eq:master_problem}. We summarize last paragraph's key content by underlining the fact that the biggest the value of $b_k/a_k$ will be, the biggest the optimal inexactness at iteration $k\in[N]$, according to our master problem, will be as well.}
$$ \rho(k_{1}) < \rho(k_{2}) \Rightarrow \nu_{k_{1}} = \frac{b_{k_{1}}}{a_{k_{1}}} \geq \frac{b_{k_{2}}}{a_{k_{2}}} = \nu_{k_{2}} \Rightarrow \delta^*_{k_{1}} \geq \delta^*_{k_{2}}$$

\renewcommand{\arraystretch}{1.5}
We are now ready to state \edited{a} general theorem about master problem \eqref{eq:master_problem}.
\newpage 
\begin{theorem}
\label{general_theorem} 
Let Assumptions A and B hold with $(a,b) \in \mathbb{R}^{N \times 2}_{++}$, $m < 1 < M$ and $h : [m \bar{\delta},M \bar{\delta}] \to \mathbb{R}_+$ being convex.
$\exists N_{\oplus},N_{\ominus} \in \{0,\dots,N-1\}$, $\lambda^* \in \mathbb{R}$ such that $\forall k \in \{0,\dots,N-1\}$, 
\begin{equation}
 \delta_{k}^* = \begin{cases} M\,\bar{\delta} & k \in \oplus := \{ \tilde{k} \,|\, \rho(\tilde{k}) < N_{\oplus}\} \\ (h^{'})^{(-1)}\Big(\frac{a_k\,\lambda^*}{b_k}\Big) & k \in \mathcal{T} := \{\tilde{k}\,|\,N_{\oplus}\leq \rho(\tilde{k}) \leq N-1-N_{\ominus}\} \\ m\,\bar{\delta} & k\in \ominus := \{  \tilde{k} \,|\, \rho( \tilde{k}) > N-1-N_{\ominus}\} \end{cases}
\label{eq:general_optimal_schedule}
\end{equation}
where $\rho(k)$ depicts the descending rank of $\nu_k = b_k / a_k$, $\lambda^*$ satisfies the equality
\begin{equation}
\sum_{k =0}^{N-1}\,b_k\,h(\bar{\delta}) - \Bigg[h(M\bar{\delta})\Bigg(\sum_{k \in \oplus}\,b_{k}\Bigg) + h(m\bar{\delta})\Bigg(\sum_{k\in \ominus}\,b_{k}\Bigg)\Bigg] = \sum_{k \in \mathcal{T}}\,b_{k}\,h\bigg((h^{'})^{(-1)}\Big(\frac{a_k\,\lambda^*}{b_k}\Big)\bigg)   
\label{eq:budget_general}
\end{equation}
and $\delta^*$ stands as a solution of \eqref{eq:master_problem}.
\end{theorem}
\begin{proof}
The proof of Theorem \ref{general_theorem} is given in Appendix \ref{proof_theorem2}.
\end{proof}

\remark If all the entries of $\nu \in \mathbb{R}^N_{++}$ differ, their ordering is unique, so becomes $\delta^*$ as shown in Appendix \ref{uniqueness_explanation}. 
\remark \label{relativity} One only need to know $a_k$ (\edited{respectively} $b_k$) up to a common factor $K_a > 0$ (\edited{respectively} $K_b > 0$). That is, it is sufficient to know $\tilde{a}_k$ (\edited{respectively} $\tilde{b}_k$) such that for any $k \in [N]$, $a_k = K_a \,\tilde{a}_k$ (\edited{respectively} $b_k = K_b\,\tilde{b}_k$). In Theorem \ref{general_theorem}, instead of looking for $\lambda^*$, one then searches for another constant $\tilde{\lambda}^* = \lambda^*\,K_a / K_b$ that would act like $\tilde{\lambda}^* \,\tilde{a}_k / \tilde{b}_k = \lambda^*\,a_k / b_k$.

\example
\label{toy} (\emph{toy example}) We want here to give a first glimpse about the upcoming optimal schedules that will apply for (among others) situations reflected in Example \ref{iAFB}, \ref{example_detailed} and \ref{example_detailed_2}. To this purpose, we clarify the above notations by writing explicitly what they entail given an academic toy example rightfully meeting Assumptions A and B. We consider that the \emph{impact coefficients} and \emph{relative costs} of oracles are given for any $k \in \{0,\dots,N-1=79\}$ by 
    $$a_k = (k+1) \hspace{20pt} b_k = \begin{cases} \frac{3}{420} & 0 \leq k < 20 \\ \frac{2}{420} & 20 \leq k < 40 \\ \frac{8}{420} & 40 \leq k < 80 \end{cases}$$
The \emph{reference} inexactness parameter is chosen as $\bar{\delta} = 10^{-4}$ and $m=0 < 1 < M = \edited{2}$. \\The oracle cost $h:[0,1] \to [0,\infty)$ is convex, differentiable and fluctuates poly-logarithmically with $0\le \delta \leq 1$ as $ \delta \to h(\delta) = \log^2(\delta^{-1})$. $h'$ is negative and strictly increasing from $[0,1]$ to $(-\infty,0]$, its inverse  $(h')^{(-1)}: (-\infty,0] \to [0,1]$ exists for $-\omega < 0$ $$-\omega \to (h')^{(-1)}(-\omega) = 2\,\omega^{-1}\,\mathcal{W}_0(\omega / 2) $$ 
where $\mathcal{W}_0$ depicts the Lambert $\mathcal{W}$ function on its $0$-branch. Here, Theorem \ref{general_theorem} applies. Obviously, $\lim_{\delta \to 0}\,h(\delta) = \infty$ translates to $\ominus = \emptyset \Leftrightarrow N_\ominus = 0$. From that point, we can efficiently solve KKT conditions. They inform that for our present problem instance, $N_\oplus = 10$. Fortunately, the indices $k \in [80]$ for which $\nu_k$ values are the $N_\oplus$ biggest fall in $\oplus = \{0,\dots,9\}$. We can conclude that the set $\mathcal{T}$ contains the indices $\{10,\dots,79\}$.\\We summarize the calculated optimal schedules for \eqref{eq:master_problem} in \eqref{eq:general_optimal_schedule_effective} bearing in mind that $\lambda = -\lambda^* \simeq \edited{27.5757}$. Figure \ref{fig:beautiful} graphs the optimal inexactness schedule in \eqref{eq:general_optimal_schedule_effective}. 
\begin{equation}
 \delta_{k}^* = \begin{cases} 2\cdot10^{-4} & k \in \{0,\dots,9\} \\ \big(\frac{2\,(3/420)}{(k+1)\,\lambda}\big)\mathcal{W}_{0}\big(\frac{(k+1)\,\lambda}{2\,(3/420)}\big) & k \in \{10,\dots,19\}\\
 \big(\frac{2\,(2/420)}{(k+1)\,\lambda}\big)\mathcal{W}_{0}\big(\frac{(k+1)\,\lambda}{2\,(2/420)}\big) & k \in \{20,\dots,39\}\\
  \big(\frac{2\,(8/420)}{(k+1)\,\lambda}\big)\mathcal{W}_{0}\big(\frac{(k+1)\,\lambda}{2\,(8/420)}\big) & k \in \{40,\dots,79\}
 \end{cases}
\label{eq:general_optimal_schedule_effective}
\end{equation}
\vspace{-10pt}
\begin{figure}
\centering
\includegraphics[width=0.8\textwidth]{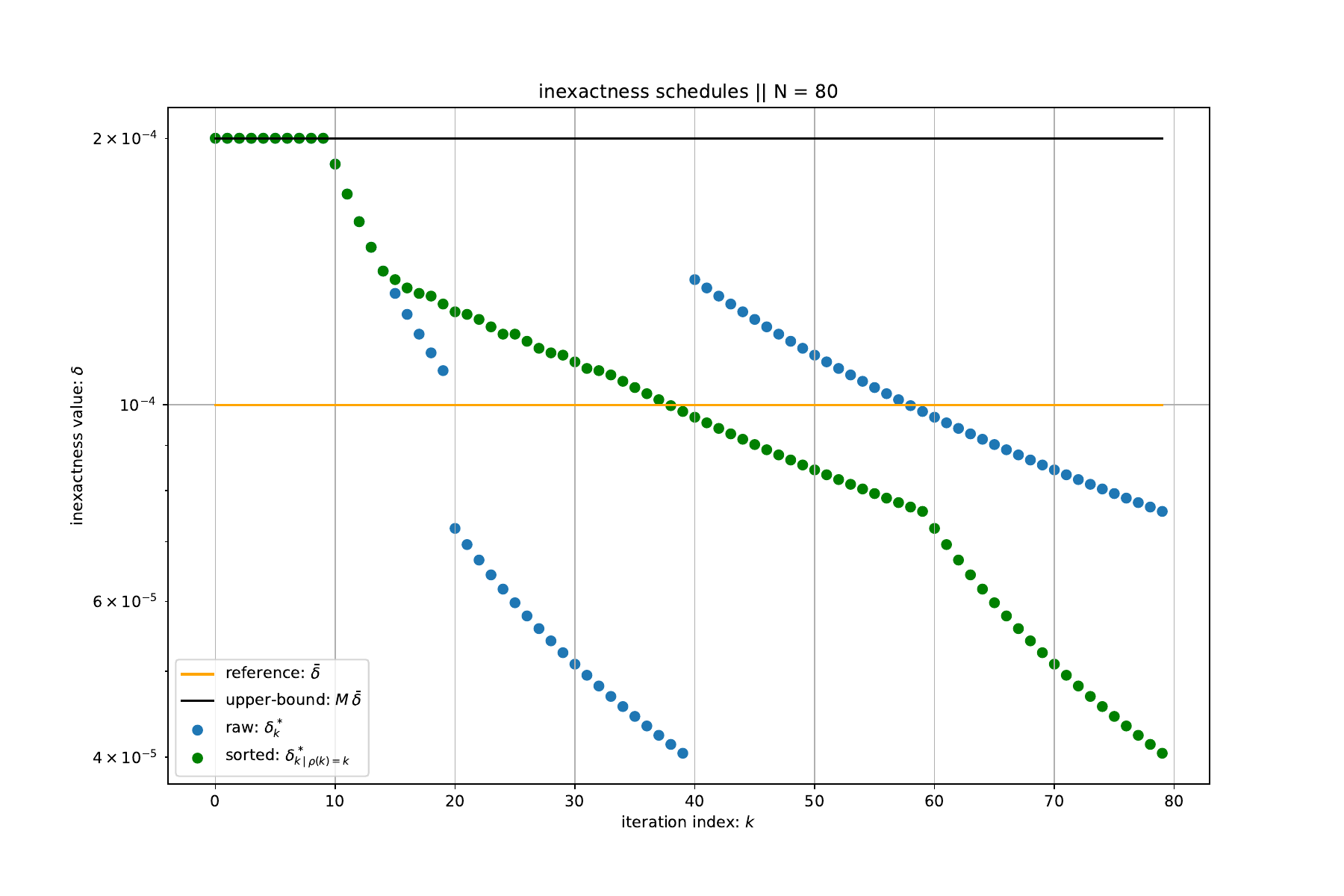}
\caption{inexactness schedules: \emph{reference}, raw optimal $\{\delta_k^*\}_{k=0}^{N-1}$ and sorted optimal $\{\delta_{\tilde{k}\,|\,\rho(\tilde{k}) = k}^*\}_{k=0}^{N-1}$.}
\label{fig:beautiful}
\end{figure}

\subsubsection{Closed-form solutions} In this section we focus on analytical closed-form solutions we can obtain from the previous theorems when specifying a certain type of $h$ function, related to the practical examples motivating our oracle cost model \eqref{eq:cost_model}. Yet, it usually remains to seize the correct values of $N_{\oplus}$ (number of iterations either performed at worst oracle accuracy or involving the least computational efforts) and $N_{\ominus}$ (number of iterations either achieved with the best oracle accuracy or demanding the heaviest computational efforts). 

\vspace{10pt}
To circumvent a cautious search for the right pair $(N_{\oplus},N_{\ominus}) \in \{0,\dots,N-1\}^2$, one can immediately detect whether $N_{\oplus} = 0 = N_{\ominus}$ with a simple trial in constant time. It is essentially what we achieve in Corollaries \ref{corollary_h_r_simp} and \ref{corollary_h_r_simp_work}.  
In such circumstances, $\oplus = \emptyset = \ominus$ (Theorem \ref{general_theorem}), i.e. the \emph{transient set} $\mathcal{T}$ that normally collects the $k$-indices of iterations associated with $\nu_k$ values smaller than the $N_{\oplus}$ biggest and bigger than the $N_{\ominus}$ smallest involves here all the iterations, i.e. $\mathcal{T} = \{0,\dots,N-1\}$.\vspace{-5pt}
\newpage 
\paragraph*{Functional family} Let us formally declare a meaningful functional family of $h$ functions that captures our applications of interest.
For any $r > 0$, we define the convex function $h_{r} : \mathbb{R}_+ \to \mathbb{R}_{++}$ and the inverse of its derivative $(h'_{r})^{(-1)} : \mathbb{R}_{--} \to \mathbb{R}_+$
\begin{equation}
    h_{r}(\delta) =
    \delta^{-r}  \hspace{5pt} \Rightarrow \hspace{5pt} (h_{r}')^{(-1)}\big(-\omega\big) = 
    \Big(\frac{\omega}{r}\Big)^{-\frac{1}{r+1}}
    \label{eq:family_broad}
\end{equation} 
\vspace{-10pt}
\corollary \label{corollary_h_r_simp} 
Let Assumptions A and B hold with $(a,b) \in \mathbb{R}^{N \times 2}_{++}$, $m < 1 < M$ and $h=h_{r} : [m\,\bar{\delta},M\,\bar{\delta}] \to \mathbb{R}_+$. If $m\,\bar{\delta} \preceq \mathring{\delta} \preceq M\,\bar{\delta}$ with  
\begin{equation}
\mathring{\delta} = \bar{\delta}\cdot \Bigg(\frac{\sum_{k=0}^{N-1}\,(b_k\,a_k^r)^\frac{1}{(r+1)}}{\sum_{k=0}^{N-1}\,b_k}\Bigg)^{\frac{1}{r}} \cdot \big(b \odot a^{-1} \big)^{\frac{1}{r+1}}
    \label{eq:gracious}
\end{equation}
then $\mathring{\delta}$ is optimal for \eqref{eq:master_problem}. 

\remark\label{simple_feasibility}  Corollary \ref{corollary_h_r_simp} simply tells that if $\mathring{\delta}$ from \eqref{eq:gracious} is feasible for our master problem \eqref{eq:master_problem} under the choice $h = h_{r}$ then it must be optimal. For the sake of completeness, we also provide in Appendix \ref{full_delta} closed-form schedules in what concerns an extended family of $h_r$ functions. It includes the logarithmic model $h_0(\delta) = \log(\delta^{-1})$ from our introductory Example \ref{introductory_example} as a pathological case $r \to 0$ and $N_\oplus, N_\ominus$ are not necessarily zero. As a drawback consequence, the employed notations become heavier. 

\paragraph*{Interpretation} Rather intuitively, an iteration whose \emph{impact coefficient} $a_k$ is bigger harms more the upper-bound $\mathcal{E}(\{\delta_k\}_{k=0}^{N-1})$ on the objective gauge $\mathcal{C}(N)$ (cfr. \eqref{eq:convergence_model}) thus requires more precision, i.e. a small $\delta_k$. However its associated cost model $\mathcal{B}_k(\cdot) = b_k\,h(\cdot)$ (cfr. \eqref{eq:cost_model}) eventually reweights the accuracy of the oracle through $b_k$ according to its relative computational burden with respect to the other iterations. Hence, $\delta_k^*$ are governed by \emph{compound impacts} $\nu_k = b_k / a_k$, 
\begin{equation}
\delta_k^* \propto \bigg(\frac{b_k}{a_k}\bigg)^{\frac{1}{r+1}}
\label{eq:delta_optimal_simplified}
\end{equation}
At fixed $a,b$, as $r$ approaches $0$, the oracles appear cheap and one can gain a lot by experiencing large deviations from $\bar{\delta}$.
Conversely, we observe that if $r$ tends to infinity then the cost of oracles spikes so that the constant schedule at $\bar{\delta}$ inexactness level, feasible, becomes optimal. Indeed, at any iteration, requesting an oracle accuracy even slightly better than $\bar{\delta}$ turns out to be extremely expensive and not affordable. Thus, we must observe $\mathcal{T} = [N]$ if $r$ is large enough. Another situation in which one can easily predict that $\mathcal{T} = [N]$ or, equivalently, $N_\ominus = 0 = N_\oplus$, arises when $m = 0$ and $M = \infty$. Indeed, akin to Example \ref{toy}, the oracle cost blows up for $\delta \to m \bar{\delta} = 0$.\\ In addition, if not obliged to, one has no advantage to choose an inexactness parameter arbitrarily large $\delta \to M \bar{\delta} = \infty$.

\remark \label{illustration_example_bis_contd} Sticking to our main thread, \edited{let us express a possible asymptotical trend of optimal schedules of inexactness for Example \ref{iAFB} ($\mu=0$, constant $\lambda_k$)} based on \eqref{eq:delta_optimal_simplified}
$$ \delta_k^* \propto A_{k+1}^{-\frac{1}{2}} \in \mathcal{O}(k^{-1})$$
\newpage 

\paragraph*{Illustration}
We illustrate \edited{now} the application of our previous theorems on an instance closely related to Example \ref{iAFB}. Indeed, when $\mu=0$ and $\lambda_k \in \mathcal{O}(1)$ for any integer $k \geq 0$, it is known that $a_k \in \Theta(k^2)$, see e.g.\@ \cite{Stonyakin20}.  We display the evolution of quantities $N_\oplus$ (Figure \ref{fig:noplus}) and $\delta_k^*$ (Figure \ref{fig:deltas}) with oracle's complexity parameter $r$ and the factor of maximal tolerated inaccuracy $M$. Let $m=0$ and \edited{let the \emph{reference} inexactness tolerated be} $\bar{\delta}=10^{-4}$. In what follows, $N_\ominus=0$ in any situation since the oracle cost model $h_r(\delta) = \delta^{-r}$ explodes as $\delta \to 0 = m\bar{\delta}$. Accordingly, we assume for any $k\geq0$ that 
$b_k = 1$ and $ a_k = (k+1)^2$. We highlight two observations of interest.
\vspace{25pt}
\begin{enumerate}
\item[$\blacksquare$]  Figure \ref{fig:noplus}: $N_\oplus>0$ if $N$ itself is big enough and $M\bar{\delta}$ constrains our master problem \eqref{eq:master_problem}, i.e.\@ $M\to 1$. In this latter case, the computational savings on intend to invest in late iterations requiring more care must be spread out on more early iterations since the biggest $\delta_k$ from any optimal schedule $\{\delta^*_k\}_{k=0}^{N-1}$ cannot take a value that falls way above the reference $\bar{\delta}$. In other words, $M\bar{\delta}$ does not allow one to save a lot of efforts in the iterations linked with the smallest \emph{impact factors}.
\vspace{25pt}

\item[$\blacksquare$] Figure \ref{fig:deltas}: The variability in the optimal schedules for inexactness $\{\delta_k^*\}_{k=0}^{N-1}$ heavily depends on the oracle cost parameter $r$, as emphasized by the power $(r+1)^{-1}$ in the relationship \eqref{eq:delta_optimal_simplified}. Our findings discussed in the previous paragraph are graphically confirmed, e.g. when $r=50$, $\delta_k^* \simeq \bar{\delta}$ for every $k\in\{0,\dots,N-1\}$.
\end{enumerate}
\begin{figure}[h]
\centering
\includegraphics[width=0.85\textwidth]{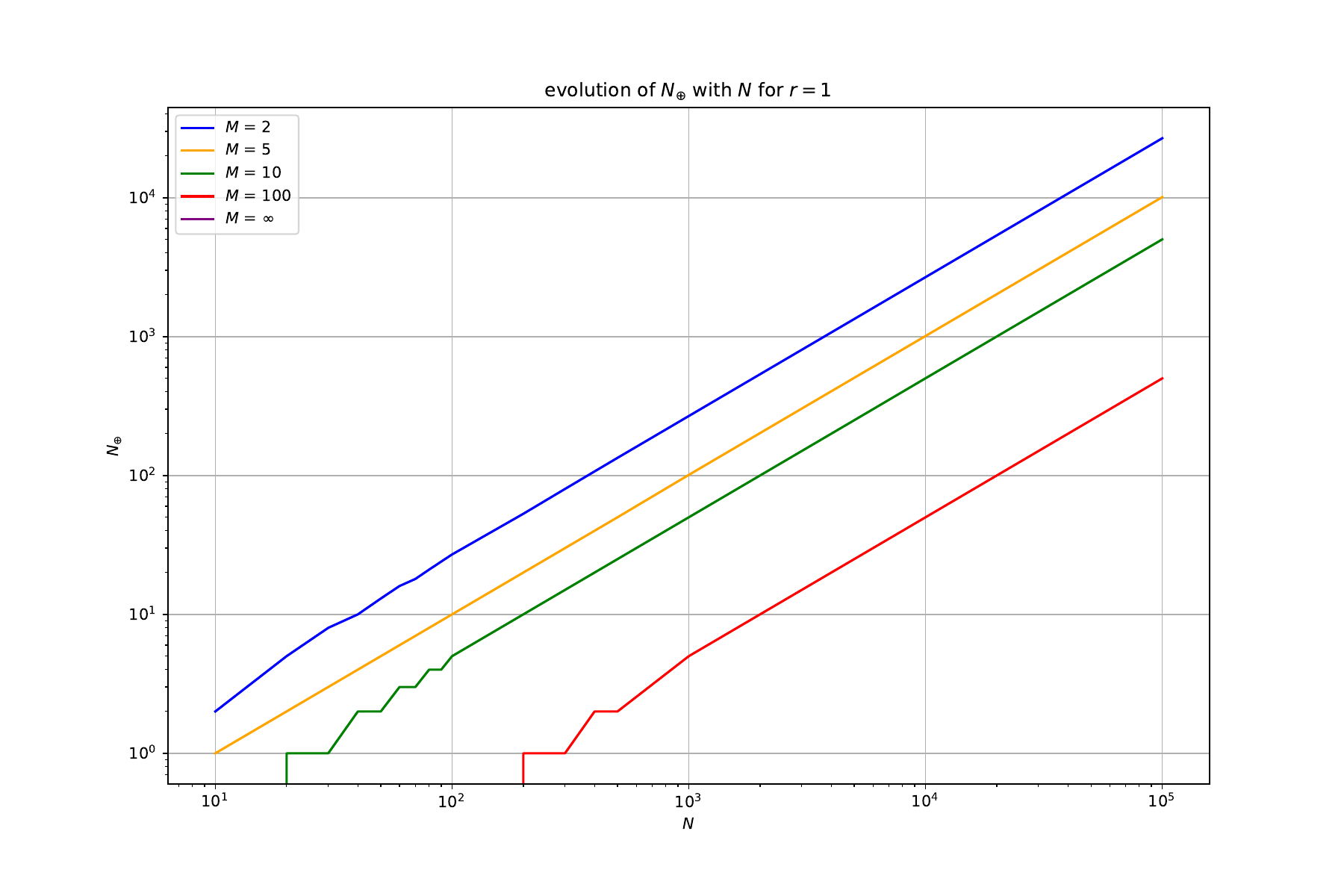}
\caption{$N_\oplus$ grows with \emph{selected} $N$, this effect starts earlier when \eqref{eq:master_problem} is more constrained, i.e.\@ $M \to 1$.}
\label{fig:noplus}
\end{figure}
\vspace{-20pt}
\begin{figure}
\centering
\includegraphics[width=0.85\textwidth]{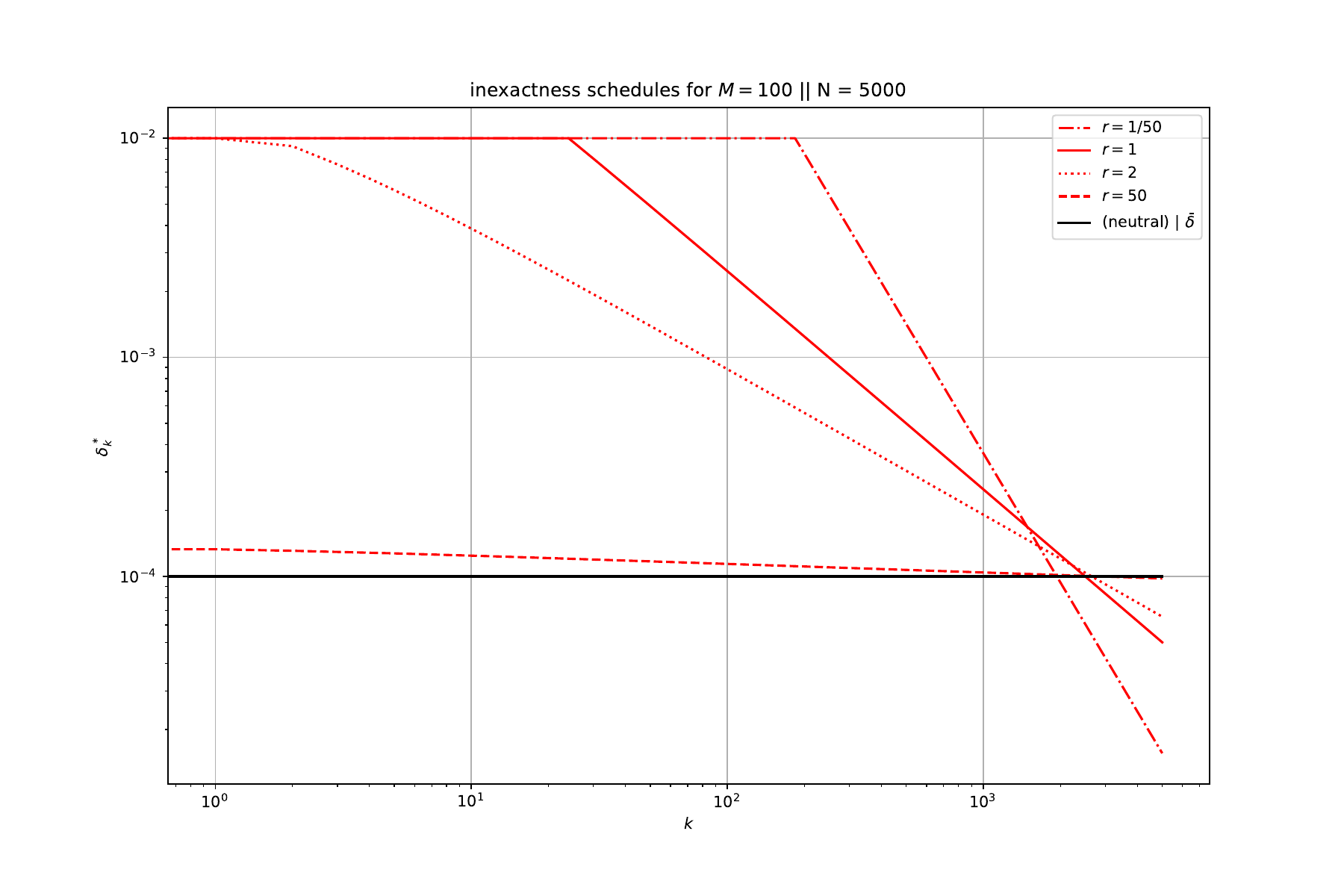}
\caption{$N$ \emph{fixed}, $r \to 0$ means cheaper oracles and thus more aggressive schedules (see $r = 50^{-1}$).}
\label{fig:deltas}
\end{figure}
\newpage

\subsection{Practical extensions}
This part of our work is dedicated to direct extensions of the results unveiled so far.\\
We adapt them to practical scenarios beyond the initial scope of TOM. Firstly, we investigate the modifications one should undertake to apply the concept of tunable oracle when, instead of the oracles' accuracies one would like to monitor the computational work invested in producing their outputs. Secondly and as previously \edited{hinted}, we propose an \emph{online} strategy which preserves the structure of optimal inexactness schedules without the knowledge of $N$.
\subsubsection{Work controlled} In some cases, one would like to manually specify the time spent at each iteration. 
In other words, instead of deciding $\delta_k$ that would \emph{a priori} lead to a cost of $\mathcal{B}_k(\delta_k)$, we process the other way around. We choose the amount of computations $\omega_k \in [\omega_M,\omega_m]$ and we expect to incur a level of inexactness $\delta_k \sim \mathcal{B}_k^{(-1)}(\omega_k)$. Therefore, we can equivalently fix a \emph{reference} total work $\bar{\omega}$ that acts as a surrogate for the budget term $\sum_{k=0}^{N-1}\,\mathcal{B}_k(\bar{\delta})$. 
We rewrite the objective from master problem \eqref{eq:master_problem} as $$\sum_{k = 0}^{N-1}\,a_k\,\delta_k \sim \sum_{k = 0}^{N-1}\,a_k\,\mathcal{B}^{(-1)}_{k}(\omega_k) = \sum_{k=0}^{N-1}\,a_k\,h_{r}^{(-1)}\big(\omega_k / b_k\big)$$

\hspace{-12pt}Notice the homogeneity of $h_r$ allows to write for any $\eta > 0$, $\beta\geq 0$, $r>0$ and $\Omega \subseteq \mathbb{R}_+^{N}$, 
\begin{equation}
\arg \min_{\omega \in \Omega}  \sum_{k=0}^{N-1}\,h_{r}^{(-1)}\big(\omega_k\cdot \eta / b_k\big) = \arg \min_{\omega \in \Omega} \sum_{k=0}^{N-1}\,h_{r}^{(-1)}\big(\omega_k / b_k\big)
\label{eq:equivalency}
\end{equation}
\newpage
This ensures that $b_k$  must only be known up to a common constant multiplicative factor, as previously assumed. We formulate the work controlled counterpart of \eqref{eq:master_problem} in \eqref{eq:master_problem_manual}, assuming that $h = h_{r}$
\begin{equation}
\edited{\omega^*(N,\bar{\omega},\omega_m,\omega_M)} \in \arg \min_{\omega_{M}
\, \preceq \, \omega \,\preceq\, \omega_{m}}
\sum_{k=0}^{N-1}\,a_k\,h^{(-1)}\Big(\frac{w_k}{b_k}\Big) \hspace{5pt}\text{s.t.}\hspace{5pt} \sum_{k=0}^{N-1}\,\omega_k = \bar{\omega}
\label{eq:master_problem_manual}
\end{equation}

\hspace{-12pt}Let us state the analogous version of Corollary \ref{corollary_h_r_simp} regarding the work controlled framework. Again, the interested reader can look at Appendix \ref{full_delta} that displays a full version with $N_{\oplus}$ and $N_\ominus$ not necessarily zero.

\corollary
\label{corollary_h_r_simp_work} 
Let Assumptions A and B hold with $(a,b) \in \mathbb{R}^{N \times 2}_{++}$, $\omega_M < \bar{\omega} / N < \omega_m$ and $h=h_{r} : \mathbb{R}_+ \to \mathbb{R}_+$. If $\omega_M \preceq \mathring{\omega} \preceq \omega_m$ with  
\begin{equation}
\mathring{\omega} = \bar{\omega}\cdot \frac{\big(b \odot a^{r} \big)^{\frac{1}{r+1}}}{\sum_{k=0}^{N-1}\,(b_k\,a_k^r)^{\frac{1}{(r+1)}}}
    \label{eq:gracious_work}
\end{equation}
then $\mathring{\omega}$ is optimal for \eqref{eq:master_problem_manual}.

\remark Despite being very similar, problems  \eqref{eq:master_problem} and \eqref{eq:master_problem_manual} are not perfectly equivalent in general. They do share the common goal of minimizing $\mathcal{E}$ subject to a budget constraint $\mathcal{B}$. Nevertheless, one can see \eqref{eq:master_problem_manual} as a version of \eqref{eq:master_problem} where the bounds on the achievable $\delta_k$ at iteration $k \in \{0,\dots,N-1\}$ vary. 
Indeed, we have the following bounds:
\begin{equation}
    \delta_k \in \Big[\mathcal{B}_k^{(-1)}(\omega_m),\,\mathcal{B}_k^{(-1)}(\omega_M)\Big]
    \label{eq:reached_accuracy}
\end{equation}
However, in the neutral case where $b \propto \mathbf{e}$, these bounds are constant and can be written as $[m\bar{\delta},\,M\bar{\delta}]$ for any $\bar{\delta}>0$ \edited{and well-chosen $0<m\leq M$ parameters.}
\paragraph*{Interpretation} 
Just as in \eqref{eq:delta_optimal_simplified}, it is not difficult to construe results from Corollary \ref{corollary_h_r_simp_work} by advocating that 
\begin{equation}
\omega_k^* \propto \bigg(b_k\,a_k^{r}\bigg)^{\frac{1}{r+1}}
\label{eq:work_optimal_simplified}
\end{equation}
We have already seen that when $a_k$ grows, linked oracle's accuracy $\delta_k$ should evolve inversely proportional. As a consequence, the computational cost $\omega_k$ rise accordingly. One should pay attention to the role of $b_k$ in \eqref{eq:work_optimal_simplified}. It seems like a bigger $b_k$ implies a bigger $\omega_k$, which turns out to be true. Meanwhile, a bigger $b_k$ is prone to curb the demand for highly accurate oracles in \eqref{eq:delta_optimal_simplified} hence intuitively reducing oracle's work $\omega_k$. Yet, there is no contradiction. 
Combining \eqref{eq:delta_optimal_simplified} and Assumption B, we can capture the overall marginal effect of $b_k$ in 
$$ \omega_k = b_k h(\delta_k) \propto b_k\, \big((b_k / a_k)^{\frac{1}{r+1}}\big)^{-r} = b_k^{\frac{1}{r+1}} a_k^{\frac{r}{r+1}}$$
\newpage 
Finally, let us assert that $b_k$'s \emph{weighting impact} diminishes as $r \to \infty$, $\omega_k^*$ becoming asymptotically proportional to the \emph{impact coefficient} $a_k$. Conversely when $r \to 0$, the suggested oracle cost somehow flattens. It ultimately depends only on $b_k$, constant in the absence of further knowledge (see Example \ref{iAFB} and \ref{example_detailed_2}). 

\subsubsection{Online version}
\label{online_session}

So far, we have assumed that the choice of $N$ was exogenous and well thought. Unfortunately, there is no \emph{one fits all} approach to adequately fix $N$. Usually, one runs an optimization algorithm and stops it as soon as a relative tolerance, tracked alongside the iterations, is observed. \edited{Let $\delta^*$ (respectively $\omega^*$) be an optimal solution of \eqref{eq:master_problem} (respectively \eqref{eq:master_problem_manual} for a specified $N \in \mathbb{N}$ and $h = h_r$.} Hidden behind equation \eqref{eq:gracious} \edited{(respectively \eqref{eq:gracious_work})} and emphasized by \eqref{eq:delta_optimal_simplified} \edited{(respectively \eqref{eq:work_optimal_simplified})}, one can retrieve a recursion rule linking the optimal inexactness parameters of two distinctive iterations, say $\hat{k}$ and $k$. Let $\delta_k^* \in [m\bar{\delta},M\bar{\delta}]$ then for any $\hat{k} \in [N]$, one can recover 
\begin{equation}
\delta_{\hat{k}}^* = \max\bigg\{m\bar{\delta},\min\bigg\{M\bar{\delta},\bigg(\frac{b_{\hat{k}}}{a_{\hat{k}}}\cdot\frac{a_k}{b_k}\bigg)^{\frac{1}{(r+1)}}\delta_k^*\bigg\}\bigg\}
\label{eq:delta_optimal_simplified_contd}
\end{equation}
Following the same logic, a recursion exists for $\omega_k^* \in [\omega_M,\omega_m]$,
\begin{equation}
\omega_{\hat{k}}^* = \max\bigg\{\omega_M,\min\bigg\{\omega_m,\bigg(\frac{b_{\hat{k}}\,a_{\hat{k}}^r}{b_{k}\,a_{k}^r}\bigg)^{\frac{1}{(r+1)}}\omega_k^*\bigg\}\bigg\}
\label{eq:work_optimal_simplified_contd}
\end{equation}

Therefore, in practice, one can adopt a \emph{reference} situation accounting for a lower-bound of $N_r$ iterations then compute its inherent \emph{offline} optimal schedule. For any iteration $\hat{k} \geq N_r$, one can extrapolate using the recursion rules explained right above.

\paragraph*{Rationale} The extrapolated schedules obtained with \eqref{eq:delta_optimal_simplified_contd}, \eqref{eq:work_optimal_simplified_contd} present the advantage to preserve the right ratios $\delta_{\hat{k}}^* / \delta_k^*$ and $\omega_{\hat{k}}^* / \omega_k^*$ as if we knew the number of iterations performed by our Tunable Oracles Method (TOM) when exited. In addition, it does not break our practical assumption of relative knowledge of $a_k$'s and $b_k$'s since we only involve ratios wiping out any common multiplicative constant. Finally, such \emph{online}  schedule allows for $b_k / a_k$'s that are not necessarily pre-computed. As mentioned earlier, in various adaptive methods such coefficients are defined by the final result of local line-search techniques \cite{BTB22,Stonyakin20,PD18}. This feature of  our \emph{online} strategy dramatically extends the applicability of our approach, endowing it with local information exploitation, allowing for possible \emph{non-monotonicity} in the parameters $\delta_k$ or $\omega_k$ used.

\section{Numerical Experiments}
\label{sec:numericals}
\edited{
We present three experiments that showcase both our elaborated \emph{offline} and \emph{online} techniques for Tunable Oracle Methods (TOM). We emphasize that we are primarily concerned with showing that our approach effectively improves the efficiency of known existing methods such as FGM \cite{Stonyakin20}.

The first experiment serves to validate our theoretical optimal \emph{offline} schedules for various levels of parameters, i.e.\@ $r$ (oracle cost parameter), $\bar{\delta}$ (\emph{reference} inexactness) and $N$ (number of performed iterations). 
In order to stick to our theoretical framework as much as possible, we generate inexact oracle outputs by adding artificial noise to the gradients used in FGM. 
The noise is chosen according to a simulated oracle cost one would invest to control the level of inexactness.
Confirming that our optimal schedules also perform better in a practical setting is important, as numerical optimization methods (with or without inexactness) typically perform better than their worst-case guarantees in practice. 
 Within this setting, we show the superiority of our schedules compared to a \emph{constant} schedule approach with matching overall computational cost.

In the second and third experiments, we investigate how the \emph{offline} and our heuristic \emph{online} approach behave within a real case where inexactness naturally emerges. Experiments 1,2 and 3 rely on the problem motivated in Example \ref{example_detailed_2} but Experiment 3 allows for line-search within FGM hence turning it into an \emph{online} method. \\ \\Our code is freely available on {\href{https://github.com/guiguiom/TunableOracleMethods}{GitHub}} so that one can observe that the reported results are representative of the usual performances.

\subsection*{Test Problem \& Data Generation} 
We describe a \emph{robust optimization} task involving a regularization parameter $\mu \geq0$ and data vectors $\Theta = \{\theta_i\}_{i=1}^{n}$,
$$ \hspace{50pt} \theta_i \sim \mathcal{N}\big(\mathbf{0},I_d / p\big)\hspace{30pt} \forall i \in \{1,\dots,n\}$$
a collection of $n \in \mathbb{N}$ \emph{scenarios} from $\mathbb{R}^d$ ($p>0$). 
We consider the classical problem where one minimizes the worst outcome of the regularized linear objective only over the previously seen \emph{scenarios}, i.e. one wishes to minimize a regularized objective $\frac{\mu}{2}||x||^2+\langle \theta, x\rangle$ over the unit simplex, i.e. $\Delta := \{x \in \mathbb{R}^d\,|\,x\succeq \mathbf{0},\,||x||_1 = 1\}$ for a meaningful cost $\theta$ based on historical data about its value, stored in $O = [\theta_1,\dots,\theta_n]^T$.
\begin{equation}
\min_{x \in \Delta}\,\max_{\theta \in \Theta}\,\langle \theta, x\rangle + \frac{\mu}{2}||x||^2 = \min_{x \in \Delta}\,\max_{\theta \in \textbf{conv}(\Theta)}\,\langle \theta, x\rangle + \frac{\mu}{2}||x||^2
\label{eq:synthetic_true}
\end{equation}

\subsection{Experiment 1 | Softmax Optimization under Synthetic Noise}
In a large scale setting, solving \eqref{eq:synthetic_true} using the standard epigraph reformulation might be prohibitive and is sometimes replaced by a smoothed version of the robust objective, leaving the feasible set intact, i.e. no extra constraint. Let $\upsilon>0$, the problem becomes\footnote{Note that the value of \eqref{eq:synthetic} differs at most by $\epsilon = \mu + n/\upsilon$ from \eqref{eq:synthetic_true}.}
\begin{equation}
F^* = \min_{x \in \mathbb{R}^d}\,\overbrace{\underbrace{\upsilon^{-1}\,\log\Big(n^{-1}\,\sum_{i=1}^{n}\,\text{exp}(\upsilon\cdot\langle \theta_i, x\rangle)\Big) + \frac{\mu}{2}||x||^2}_{f(x)} + \underbrace{\chi_{\Delta}(x)}_{\Psi(x)}}^{F(x)} 
\label{eq:synthetic}
\end{equation}
Solving \eqref{eq:synthetic} using FGM requires the ability to compute (in-)exact information about $f$ at any query point $x \in X:= \Delta$, i.e. $(\tilde{f}(x),\nabla \tilde{f}(x)) \simeq (f(x),\nabla f(x))$. Here, the value of $f(x)$ and its gradient $\nabla f(x)$ are available in closed-form for \eqref{eq:synthetic} hence discarding the need for schedules of inexactness. However, one could think about a game where one has to pay $\delta^{-r}$ ($r>0$) or $-\log(\delta)$ ($r=0$) to obtain $\nabla f(x)$ corrupted by a noise of limited radius $\alpha \,\delta$ depending on parameter $r\geq0$. I.e.\@ one gets $\nabla \tilde{f}(x)$ such that 
\begin{equation}
\nabla \tilde{f}(x) = \nabla f(x) + \alpha\,\delta \cdot u
\label{eq:process}
\end{equation}
with $u \sim \mathcal{U}\big(\{u \in \mathbb{R}^d\,|\,||u||=1\}\big)$. From \cite{Devolder2013FirstorderMW}, one can deduce that $(f(x), \nabla \tilde{f}(x))$ as defined above provides a $(4\,\alpha\,\delta, \upsilon \cdot ||O||^2+\mu,\mu)$ inexact information. 
\paragraph*{Test instances} As suggested by Example \ref{example_detailed_2}, let us now try out FGM on problem \eqref{eq:synthetic} with fixed stepsize $L^{-1} = (\upsilon \cdot ||O||^2+\mu)^{-1}$ and different values of $\mu$, $r$, $N$, $\bar{\delta}$ while using sequences of inexact information tuples parametrized by schedules $\{\delta_k\}_{k=0}^{N-1}$. That is, for each instance $(\mu,r,N,\bar{\delta}) \in \{0,10^{-1}\} \times \{1\} \times \{10,500,1000,5000,10000\} \times \{10^{-1},10^{-3},10^{-5}\}$ or $(\mu,r,N,\bar{\delta}) \in \{0,10^{-1}\} \times \{0\} \times \{10,500,1000,5000,10000\} \times \{9\cdot10^{-3},10^{-4},10^{-6}\}$, we compare our \emph{offline} optimized approach $\{\delta_k = \delta_k^*\}_{k=0}^{N-1}$ where $\delta^*$ solves \eqref{eq:master_problem} with vectors $(a,b)$ specified in Table \ref{tab:my_tabular_AB}, $m=0$ and $M=100$ (we allow the inexactness from our \emph{tunable} approach to be a hundred times worse than the \emph{reference} $\bar{\delta}$) with the \emph{constant} schedule $\{\delta_k = \bar{\delta}\}_{k=0}^{N-1}$. We define the simulated cost of inexact oracles as $N / \bar{\delta} = \sum_{k=0}^{N-1}\,(\delta_k^*)^{-1}$ ($r=1$) and $-N \log(\bar{\delta}) = -\sum_{k=0}^{N-1}\,\log(\delta_k^*)$ ($r=0$)\footnote{$M\,\bar{\delta} < 1$ in order to ensure that any iteration costs a strictly positive amount of work.} and match it to the total cost of the constant $\bar{\delta}$ \emph{reference} schedule. Finally, in order to tame the effects of the randomness inherent to our synthetic inexact information production, we aggregate the obtained results from FGM over $5$ runs with $5$ different starting points drawn as $x_0 \sim \mathcal{U}\big(\Delta\big)$.  Here below, the dimensions read $(d,n,\alpha,p) =(100,500,100,10)$.

\paragraph*{Estimation of $F^*$}
Regarding this simple example, $F^*$ is estimated by the value of $F$ at the output from a noise-free FGM applied for $3\cdot10^5$ iterations.

\paragraph*{Results}
Figure \ref{fig:synthetic_no_reg} and \ref{fig:synthetic_reg} show the evolution of the terminal primal optimality gap $F(x_N)-F^* = f(x_N)-F^*$ after $N$ iterations of FGM (output iterate $x_N$ is in $\Delta$) using the variants of inexactness schedules described above. Each dot plotted on these graphs represents the averaged performance over the $25 = 5 \cdot 5$ independent runs of either our \emph{tunable} or the \emph{constant} approach for a single instance $(\mu,r,N,\bar{\delta})$.\\
As predicted by the theory, the \emph{tunable} inexactness schedule reached a better terminal primal optimality gap for every setting tried. Not surprisingly, the gains were dramatically more impressive when the discrepancy within \emph{impact coefficients} $\{a_k\}_{k=0}^{N-1}$ was high. Intuitively, when such coefficients do not vary at all, our optimized \emph{tunable} approach boils down to the \emph{constant} schedule and no gain is to expect at all. Table \ref{tab:my_tabular_AB} recalls that $a_k = A_{k+1}$ for any $k \in [N]$ and Remark \ref{illustration_example} suggests that $\{A_{k}\}_{k \in \mathbb{N}}$ grows sublinearly (respectively linearly) when $\mu=0$ (respectively $\mu>0$). Therefore, settings in which $\mu>0$ are the most promising in terms of gains in favor of the \emph{tunable} approach, thanks to the inherent high variance of \emph{impact coefficients}. Figure \ref{fig:synthetic_reg} undoubtedly confirmed this hope with substantial gains for our \emph{tunable} approach although problem's conditioning stays relatively small, i.e.\@ $\mu/L \simeq 10^{-5}$. Both Figure \ref{fig:synthetic_no_reg} and \ref{fig:synthetic_reg} show that the benefits of high discrepancy in \emph{impact coefficients} (hence in optimal inexactness schedules, see \eqref{eq:delta_optimal_simplified}) are enhanced by cheap oracles, i.e.\@ $r\to 0$. We point out as a generic comment that for a small number of iterations, the additive effect of inexactness does not dominate the \emph{error-free} convergence bound and one can barely observe any difference between our optimized approach and other inexactness schedules. Finally, one can observe that some of the results linked to more accurate \emph{reference} inexactness of $\bar{\delta} \leq 10^{-5}$ led to poorer primal optimality gaps for $N = 10^2$. This unexpected phenomenon can be explained by the fact that for bigger values of $\delta$, \eqref{eq:process} might actually yield more aggressive and successful descent directions, compared to plain accurate gradients, again in a regime where the additive impact of inexactness is of secondary importance.

\begin{figure}[ht]%
    \centering
    \vspace{-20pt}
    \subfloat[\centering logarithmic oracle cost ($r=0$)]{{\includegraphics[width=7.5cm]{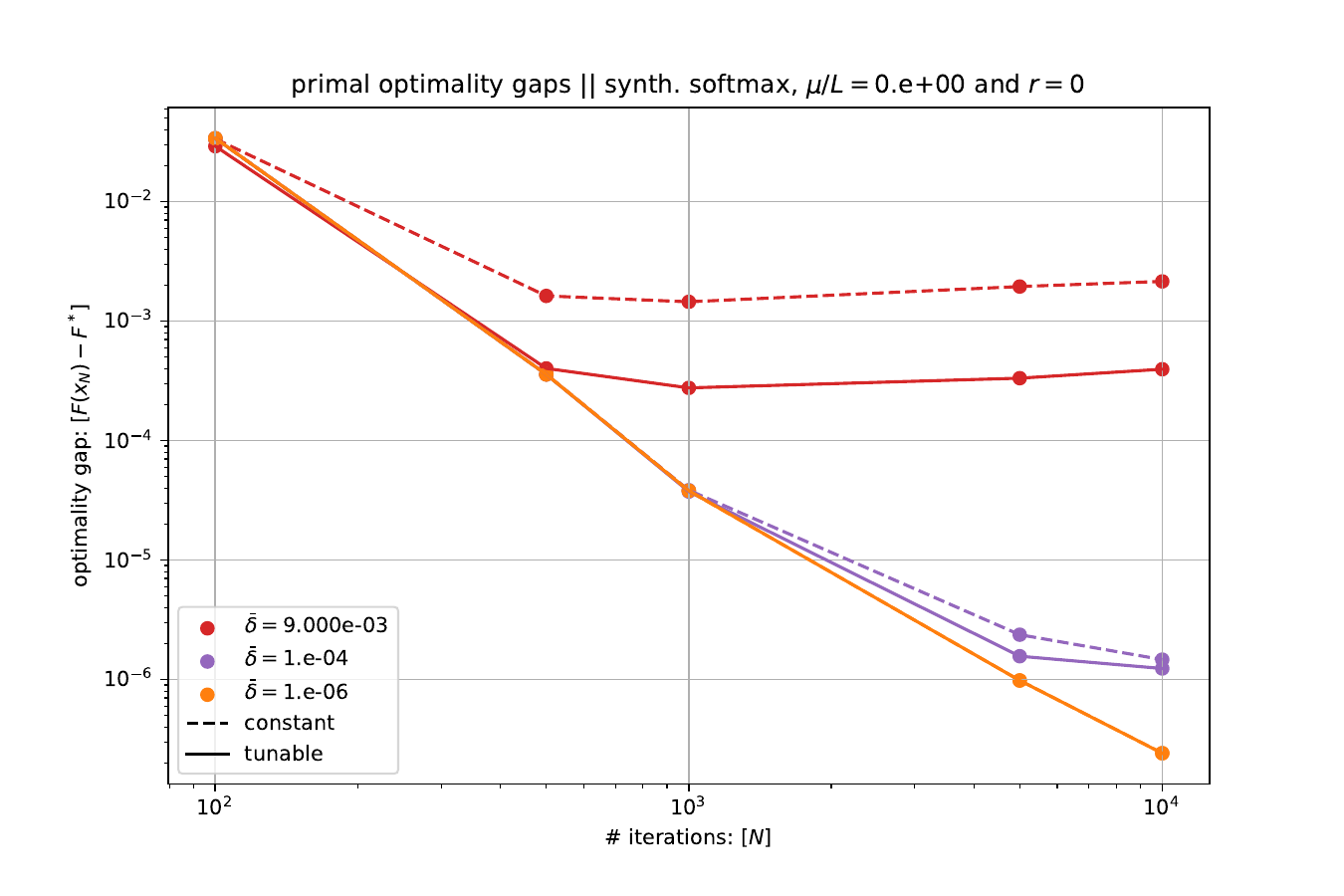} }}%
    \subfloat[\centering inversely proportional oracle cost ($r=1$)]{{\includegraphics[width=7.5cm]{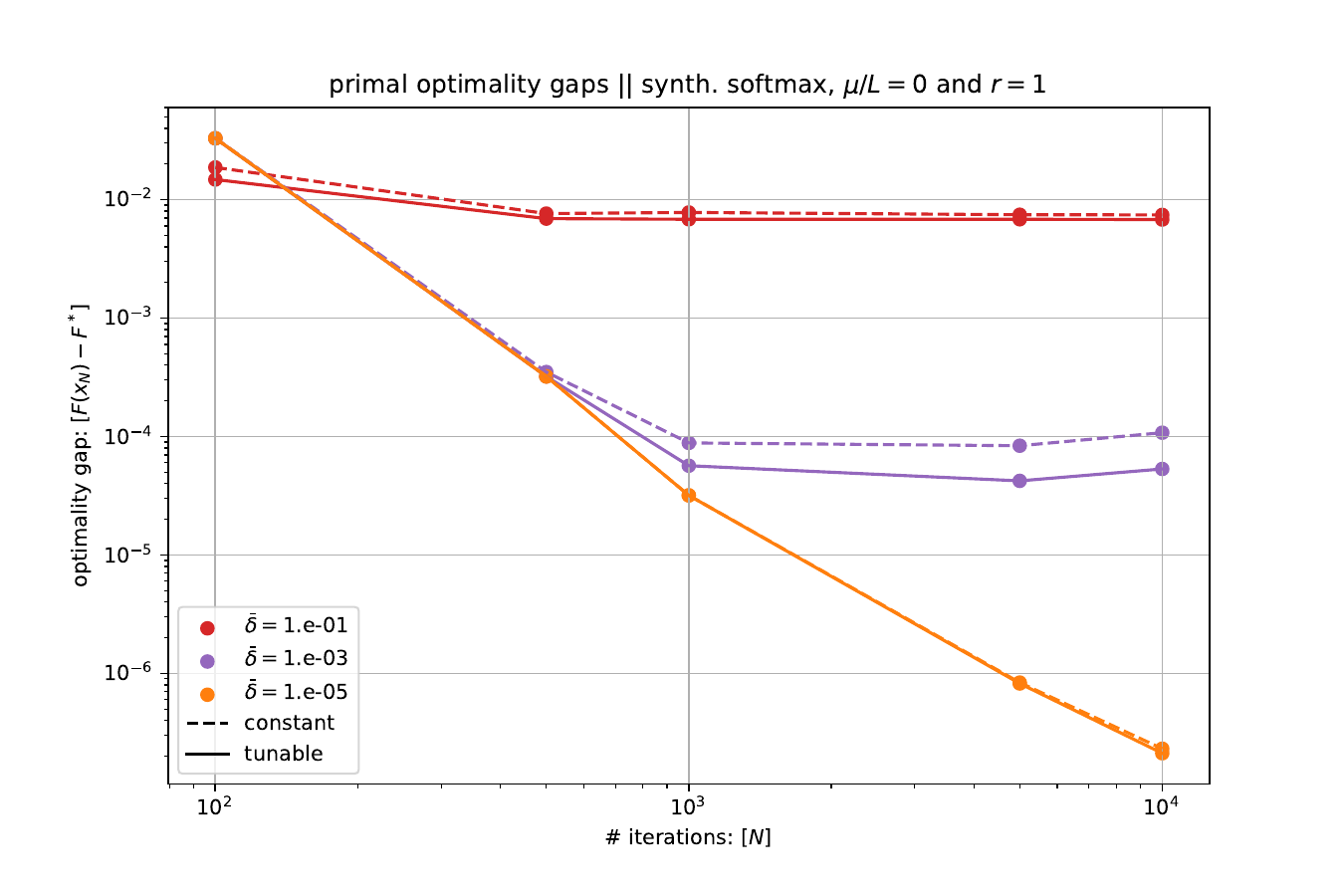} }}%
    \caption{(no regularization, $\mu=0$) | Softmax Optimization with Synthetic Noise }%
    \label{fig:synthetic_no_reg}%
\end{figure}

\begin{figure}[ht]%
    \centering
    \vspace{-20pt}
    \subfloat[\centering logarithmic oracle cost ($r=0$)]{{\includegraphics[width=7.5cm]{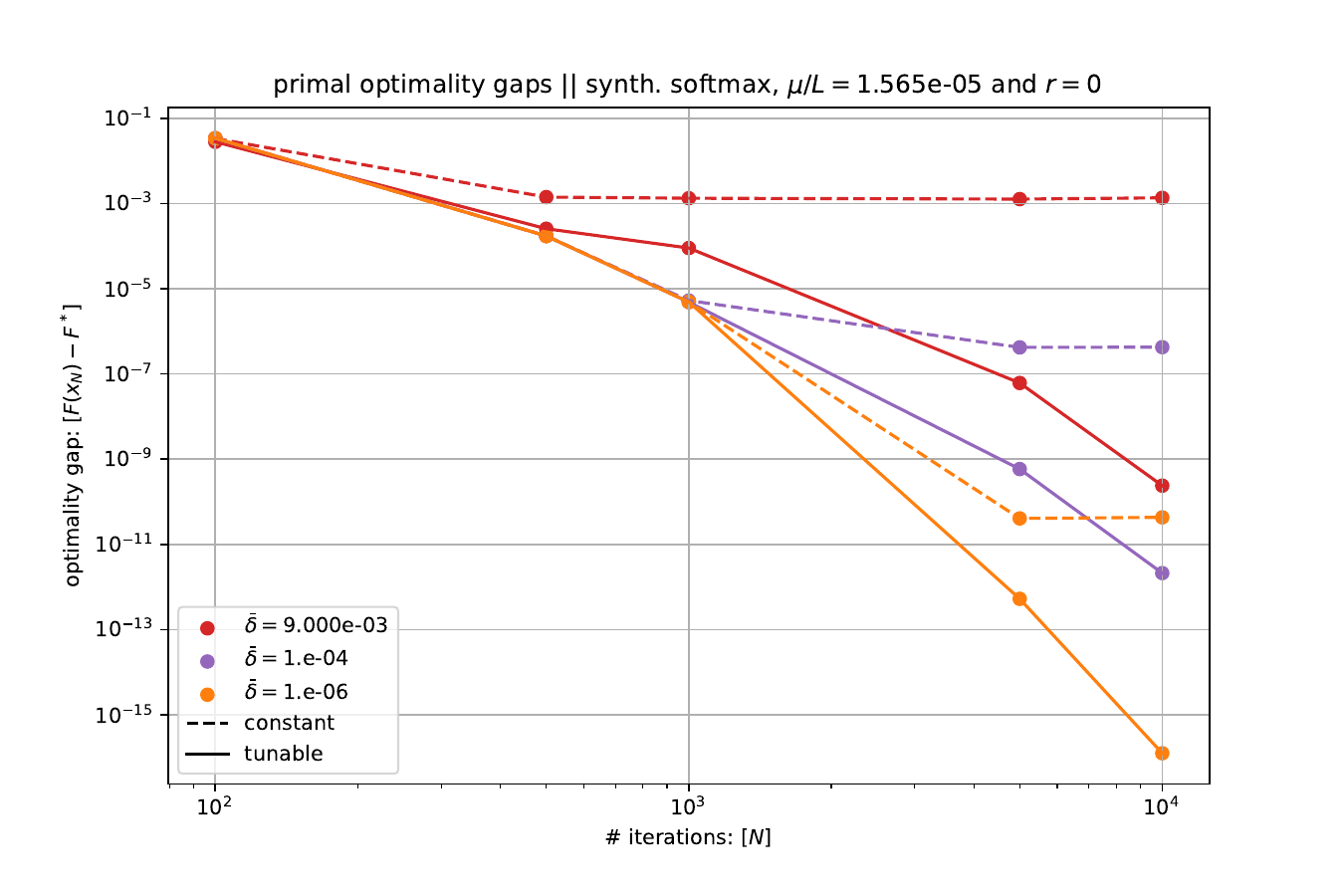} }}%
    \subfloat[\centering inversely proportional oracle cost ($r=1$)]{{\includegraphics[width=7.5cm]{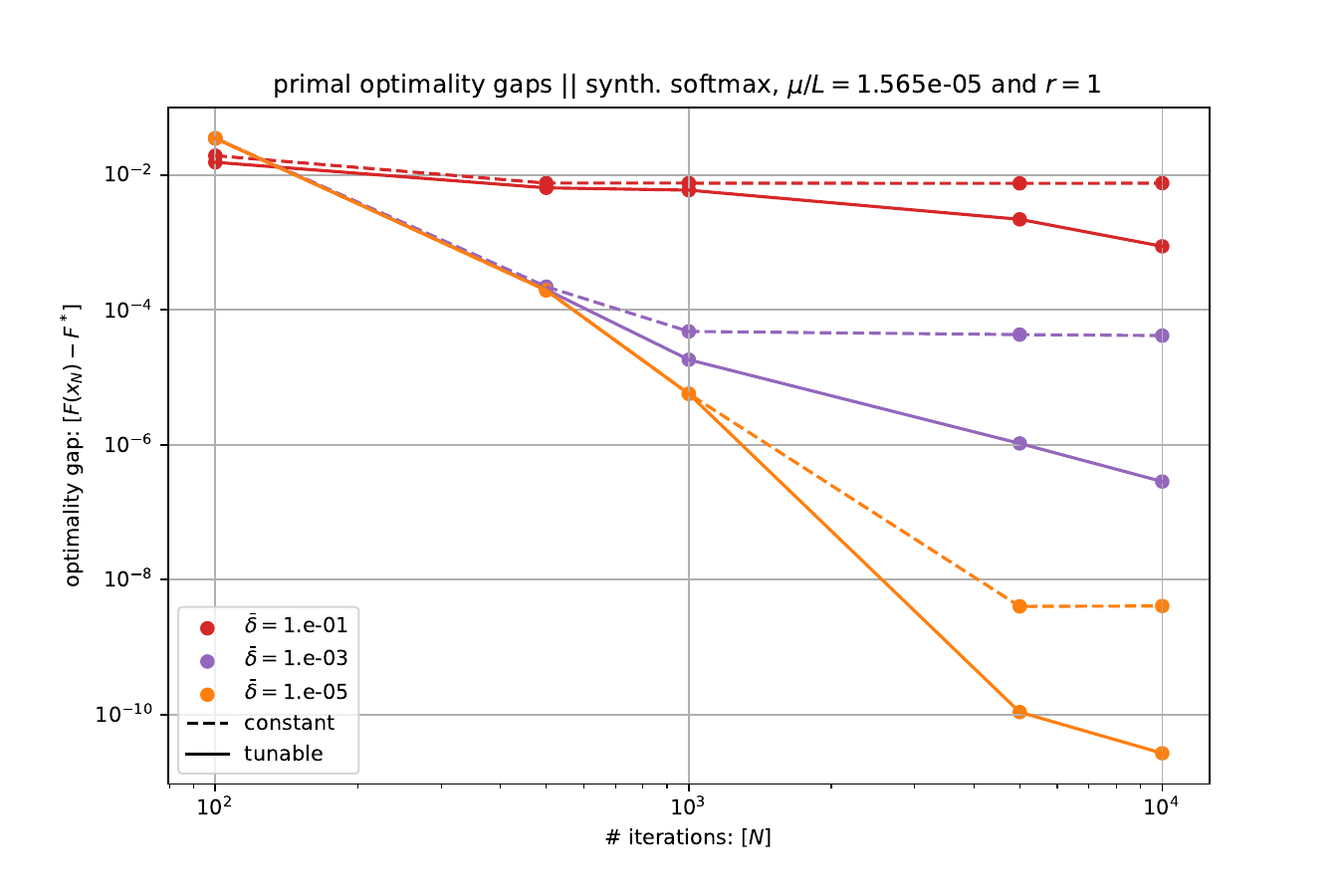} }}%
    \caption{(regularization, $\mu>0$) | Softmax Optimization with Synthetic Noise }%
    \label{fig:synthetic_reg}%
\end{figure}

\subsection{Experiment 2 \& 3 | Robust Optimization over Convex Hull}
As depicted in Example \ref{example_detailed_2}, one may choose an \emph{anchor scenario} $\bar{\theta} \in \textbf{conv}(\Theta)$ and a parameter $\sigma>0$ in order to regularize the inner maximization problem in \eqref{eq:synthetic_true} (right-hand side) yielding a smoothed outer level objective \cite{Nes05}. Then, a challenging task consists of minimizing the worst possible outcome of the inner-regularized objective over the convex hull of all the \emph{scenarios}, i.e.\@

\begin{equation}
\min_{x \in \Delta}\,\max_{\theta \in \textbf{conv}(\Theta)}\,\langle \theta, x\rangle - \frac{\sigma}{2}||\theta-\bar{\theta}||^2 + \frac{\mu}{2}||x||^2
\label{eq:offline_true}
\end{equation}
To that end, we adopt the \emph{anchor} $\bar{\theta} = n^{-1}\, \sum_{i=1}^{n}\, \theta_i$ and construct the problem 
\begin{equation}
F^* = \min_{x \in \mathbb{R}^d}\,\overbrace{\underbrace{\frac{\mu}{2}||x||^2+\max_{w\succeq \mathbf{0},\, ||w||_1=1}\,\underbrace{\langle O^Tw ,x\rangle-\frac{\sigma}{2}||O^Tw -\bar{\theta}||^2}_{q(w ; x)}}_{f(x)} + \underbrace{\chi_{\Delta}(x)}_{\Psi(x)}}^{F(x)}
\label{eq:off_paper}
\end{equation} 
Again, as motivated in Example \ref{example_detailed_2}, we use FISTA from \cite{vdb22} with a momentum depending on the number $\hat{\kappa}= \frac{\lambda_{\text{min}}(OO^T)}{\lambda_{\text{max}}(OO^T)}$ in order to get $w_x \succeq \mathbf{0}$, $||w_x||_1=1$ ensuring 
\begin{equation}
f(x) - \bigg(\frac{\mu}{2}||x||^2 + q(w_x ; x)\bigg) \leq \delta
\label{eq:offline_requirement}
\end{equation}
Then $(\tilde{f}(x),\nabla \tilde{f}(x)) = (\frac{\mu}{2}||x||^2 + q(w_x ; x) , \mu x + O^T w_x)$ provides $(2\delta,2\cdot \sigma^{-1} + \mu,\mu)$ inexact information about $f$, tuple used within FGM to solve problem \eqref{eq:offline_true}. \\
We monitor the quality of a candidate $w_{x}^{(\omega)}$ at any point $x \in \Delta$ after $\omega$ inner-iterations thanks to the Frank-Wolfe gap \eqref{eq:FW_gap}. As soon as $\text{FW}(\omega) \leq \delta$ then $w_x^{(\omega)}$ fulfills \eqref{eq:offline_requirement}.
\begin{equation} 
\text{FW}(\omega) = \max_{j \in [\omega]}\,  \max_{w\succeq \mathbf{0},\, ||w||_1=1}\,q(w_{x}^{(j)} ; x)  + \langle \nabla q(w_{x}^{(j)} ; x), w - w_x^{(j)} \rangle  - q(w_{x}^{(\omega)} ; x) 
\label{eq:FW_gap}
\end{equation}
In order to speed-up computations to obtain approximate solutions of the inner problem, we give as starting point of FISTA the solution obtained at the last oracle call.
\paragraph*{Estimation of $F^*$} Let $\hat{x}^* \in \Delta$ be the best known solution to \eqref{eq:off_paper}. Then, by convexity of the objective, we estimate a lower-bound for $F^*$ as follows
\begin{equation} 
F^* = \min_{x \in \Delta} f(x) \geq \min_{x \in \Delta}\, \underbrace{\frac{\mu}{2}||\hat{x}^*||^2 + q(w_{\hat{x}^*} ; \hat{x}^*)}_{\tilde{f}(\hat{x}^*)}  + \langle \underbrace{\mu \hat{x}^* + O^T w_{\hat{x}^*}}_{\nabla \tilde{f}(\hat{x}^*)} , x-\hat{x}^*\rangle + \frac{\mu}{2}||x-\hat{x}^*||^2
\label{eq:lower_bound_opt}
\end{equation}
where $w_{\hat{x}^*}$ represents the $w$ candidate that achieved $10^{-10}$ precision when computing an approximate pair $(\tilde{f}(\hat{x}^*), \nabla \tilde{f}(\hat{x}^*)) \simeq (f(\hat{x}^*),\nabla f(\hat{x}^*))$ with respect to criterion \eqref{eq:offline_requirement}.

\subsubsection*{Offline Schedules (Experiment 2)}
\paragraph*{Test instances}
Here, we run FGM with stepsize $L^{-1} = (2 \sigma^{-1}+\mu)^{-1}$ for a predefined number of iterations $N$. We have tested the two following settings in terms of scaling and dimension; $(d,n,p,\sigma) =(100,500,5^{-1},10^{-3})$ and $(d,n,p,\sigma) =(1000,500,5^{-1},10^{-3})$. The first setting led to $\hat{\kappa} = 0$ hence implying that the oracle cost parameter $r$ equals $\frac{1}{2}$ whereas under the second setting $\hat{\kappa} > 0$ and $r$ is taken as $0$, recalling the arguments of Example \ref{example_detailed_2}. For each instance $(\mu,r,N,\bar{\delta}) \in \{0,10^{-1}\} \times \{\frac{1}{2}\} \times \{10,500,1000,5000\} \times \{10^{-1},10^{-3},10^{-5}\}$ or $(\mu,r,N,\bar{\delta}) \in \{0,10^{-1}\} \times \{0\} \times \{10,500,1000,5000\} \times \{9\cdot10^{-3},10^{-4},10^{-6}\}$, we obtain our \emph{tunable} $\delta^*$ by solving problem \eqref{eq:master_problem} in the very same fashion as Experiment 1. Although the overall computational cost of our \emph{tunable} $\{\delta_k=\delta^*_k\}_{k=0}^{N-1}$ schedule is meant to coincide with the computational cost of the \emph{constant} schedule $\{\delta_k=\bar{\delta}\}_{k=0}^{N-1}$, the oracle cost model does not always perfectly fit the true encountered computational costs required to produce the inexact information tuples. Therefore, we record the total number of inner-iterations $\sum_{k=0}^{N-1}\,\omega_k$ as well as the final primal accuracy about \eqref{eq:off_paper} for fair comparisons. This latter entails the computation of $f(x_N)=F(x_N)$, also performed at $10^{-10}$ precision with respect to criterion \eqref{eq:offline_requirement} (see above). Here, $\hat{x}^*$ stems as as the output of FGM for $8 \cdot 10^3$ iterations with \emph{constant} inexactness level of $10^{-8}$ when $r=0$. Because of much more expensive oracles, we execute only $5\cdot 10^3$ iterations at inexactness level $10^{-7}$ to produce $\hat{x}^*$ when $r=\frac{1}{2}$.
To reduce the eventual initialization bias $x_0 \sim \mathcal{U}(\Delta)$, we averaged the results of $10$ repetitions.\\

\paragraph*{Results} Here, we have plotted the terminal primal optimality gaps against the overall oracle workloads $\sum_{k=0}^{N-1}\,\omega_k$ after $N$ (outer-)iterations of FGM. Since the wall-clock times strongly correlate with the workloads, we have chosen no to display duplicate figures other than Figure \ref{fig:rlp_no_reg} and \ref{fig:rlp_reg}. Again, each dot symbolizes the averaged performances tracked for one specific inexactness schedule for one specific instance defined by the vector $(\mu,r,N,\bar{\delta})$. One can observe that most of the trends highlighted in the paragraph about the results of our synthetic first experiment still hold for this real second experiment. Notably, it looks like that regularization plays a key role in the substantial superiority of our optimized \emph{tunable} approach. Again, this is understood by a variability of possibly several orders of magnitude in the value of \emph{impact coefficients} $\{a_k\}_{k=0}^{N-1}$ when $\mu>0$. In every case, our \emph{tunable} approach was more efficient, i.e.\@ it dominated the \emph{constant} approach in the Pareto sense in the plane of reached primal optimality versus unit of workload spent in the inner problems. On Figure \ref{fig:rlp_no_reg} (a) and \ref{fig:rlp_reg} (a), the \emph{tunable} schedule based on a \emph{reference} inexactness $\bar{\delta}$ fixed to $10^{-6}$ (yellow) did take more inner-iterations for $N=500$ than $N=1000$. Although not intuitive, this behaviour can be explained. It happens that oracles cost way less than expected, especially for low requested accuracies. In what concerns Experiment 2,\\ it could even happen that FISTA's starting iterate at outer-iteration $k$, $w^{(0)}_{x_{k}} = w^{(\omega_{k-1})}_{x_{k-1}}$, already fulfills $\delta_k$ inexactness for free, i.e. $\omega_k = 0$. It turns out that our \emph{tunable} approach for $N=1000$ benefited more of such \emph{cheap-meal} phenomenons since its linked schedule demands less accurate oracles than the one with $N=500$ in early iterations.
Overall, a \emph{reference} inexactness $\bar{\delta}\in [10^{-3},10^{-2}]$ provided the best \emph{trade-off} to reach a given primal accuracy target within this \emph{offline} setting.
\vspace{5pt}
\begin{figure}[ht]%
    \centering
    \vspace{-20pt}
    \subfloat[\centering smooth strongly convex inner problem ($r=0$)]{{\includegraphics[width=7.5cm]{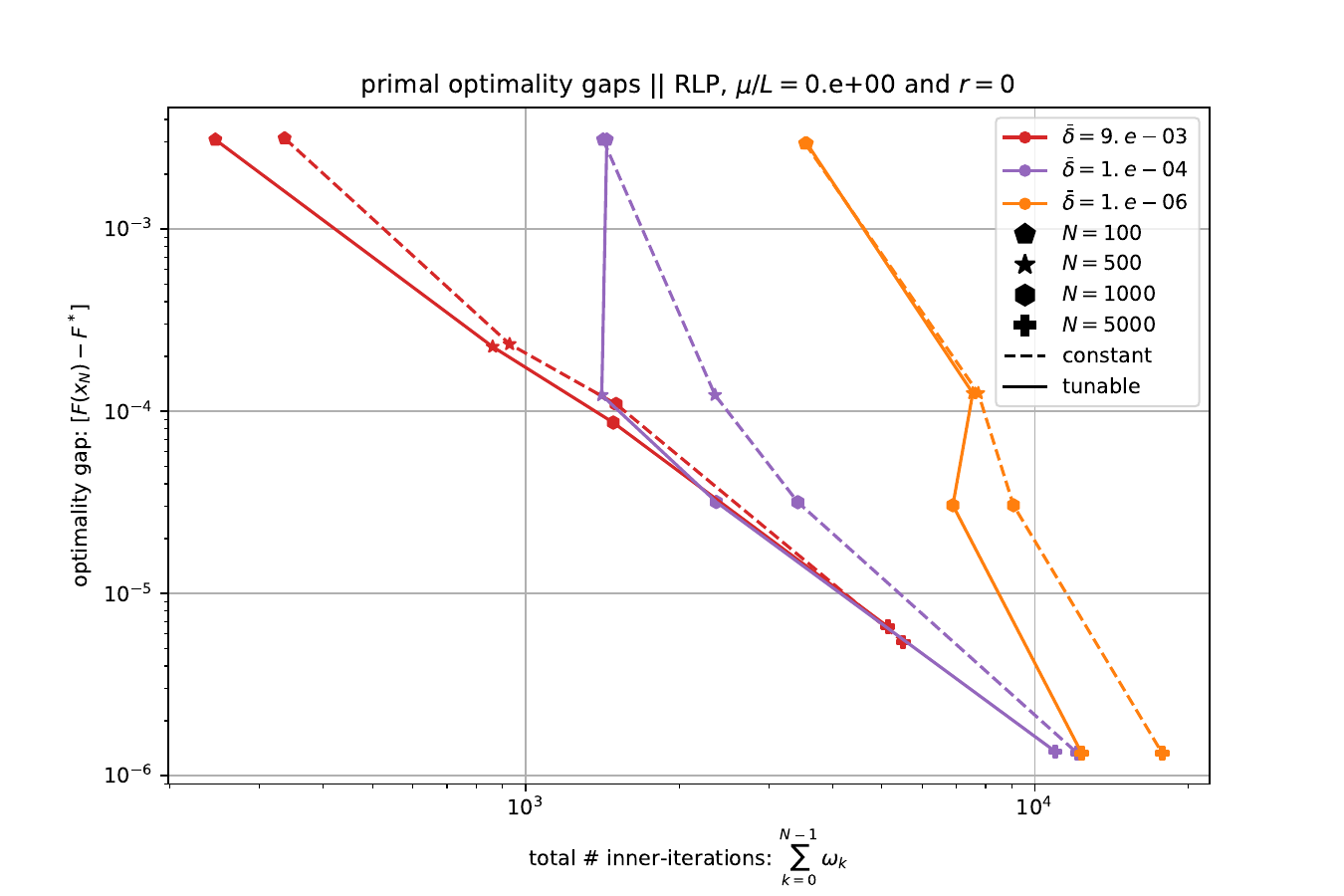} }}%
    \subfloat[\centering smooth convex inner problem ($r=\frac{1}{2}$)]{{\includegraphics[width=7.5cm]{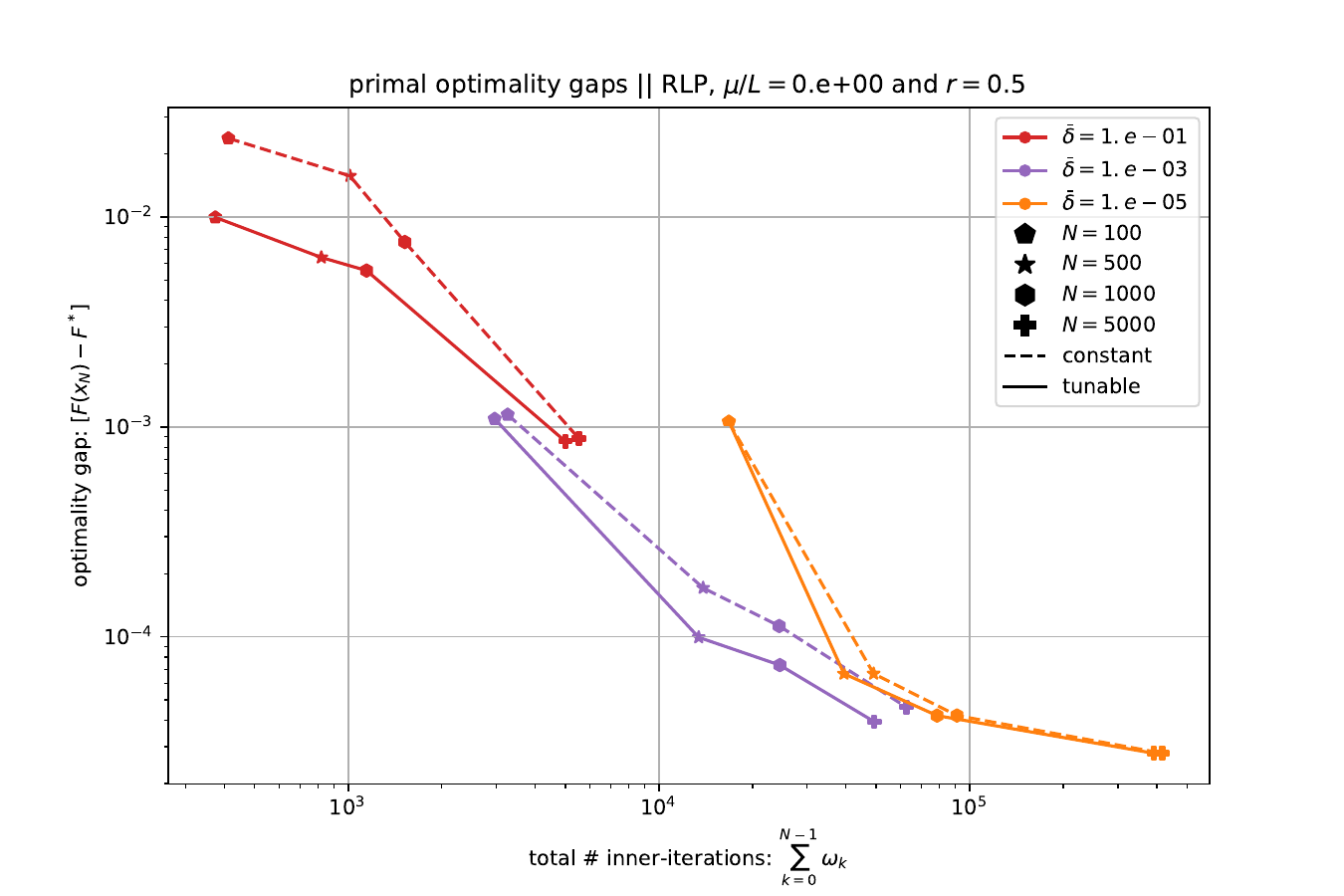} }}%
    \caption{(no regularization, $\mu=0$) | R.O. over Convex Hull, \emph{offline} with multiple $\bar{\delta}$}%
    \label{fig:rlp_no_reg}%
\end{figure}

\begin{figure}[ht]%
    \centering
    \vspace{-20pt}
    \subfloat[\centering smooth strongly convex inner problem ($r=0$)]{{\includegraphics[width=7.5cm]{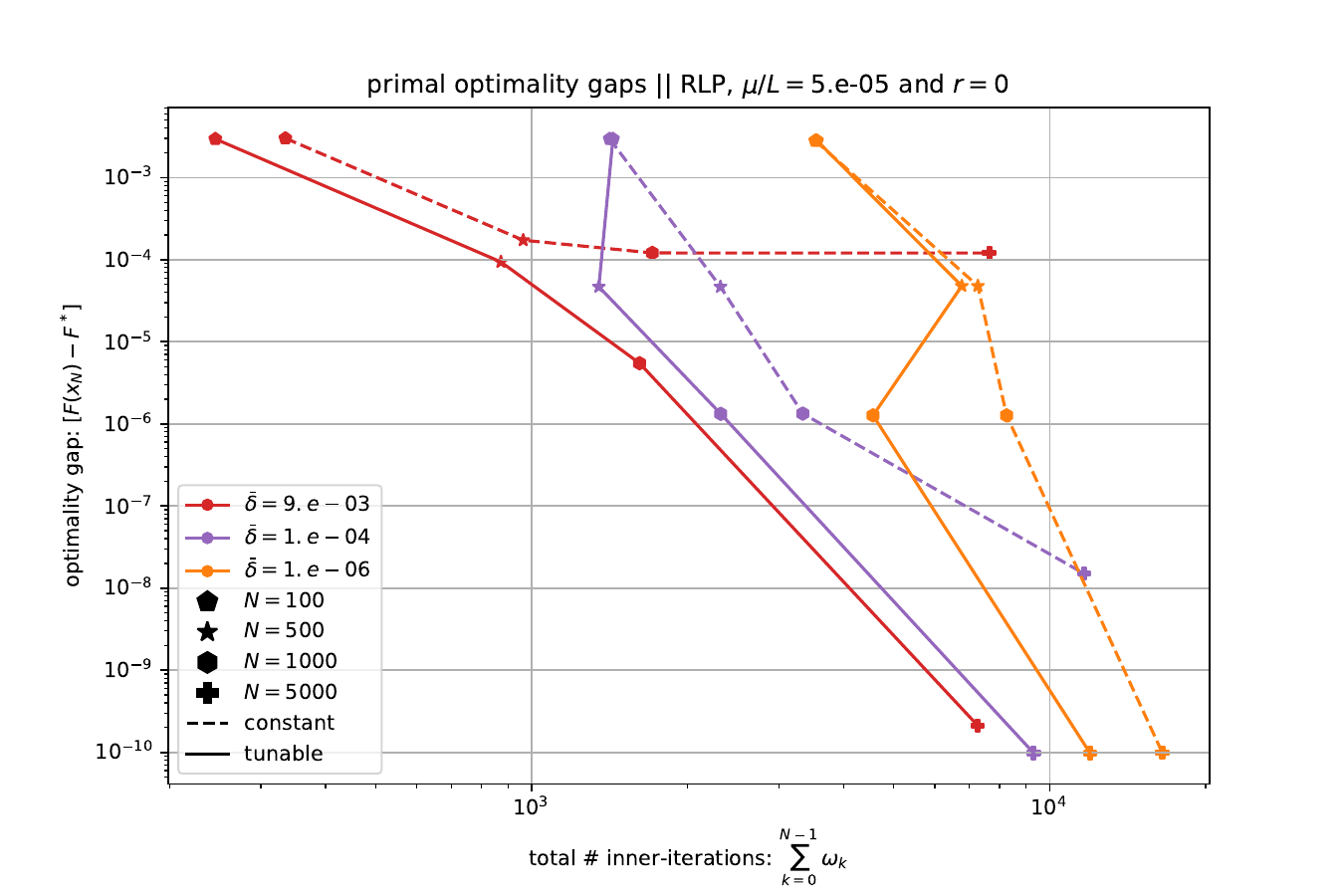} }}%
    \subfloat[\centering smooth convex inner problem ($r=\frac{1}{2}$)]{{\includegraphics[width=7.5cm]{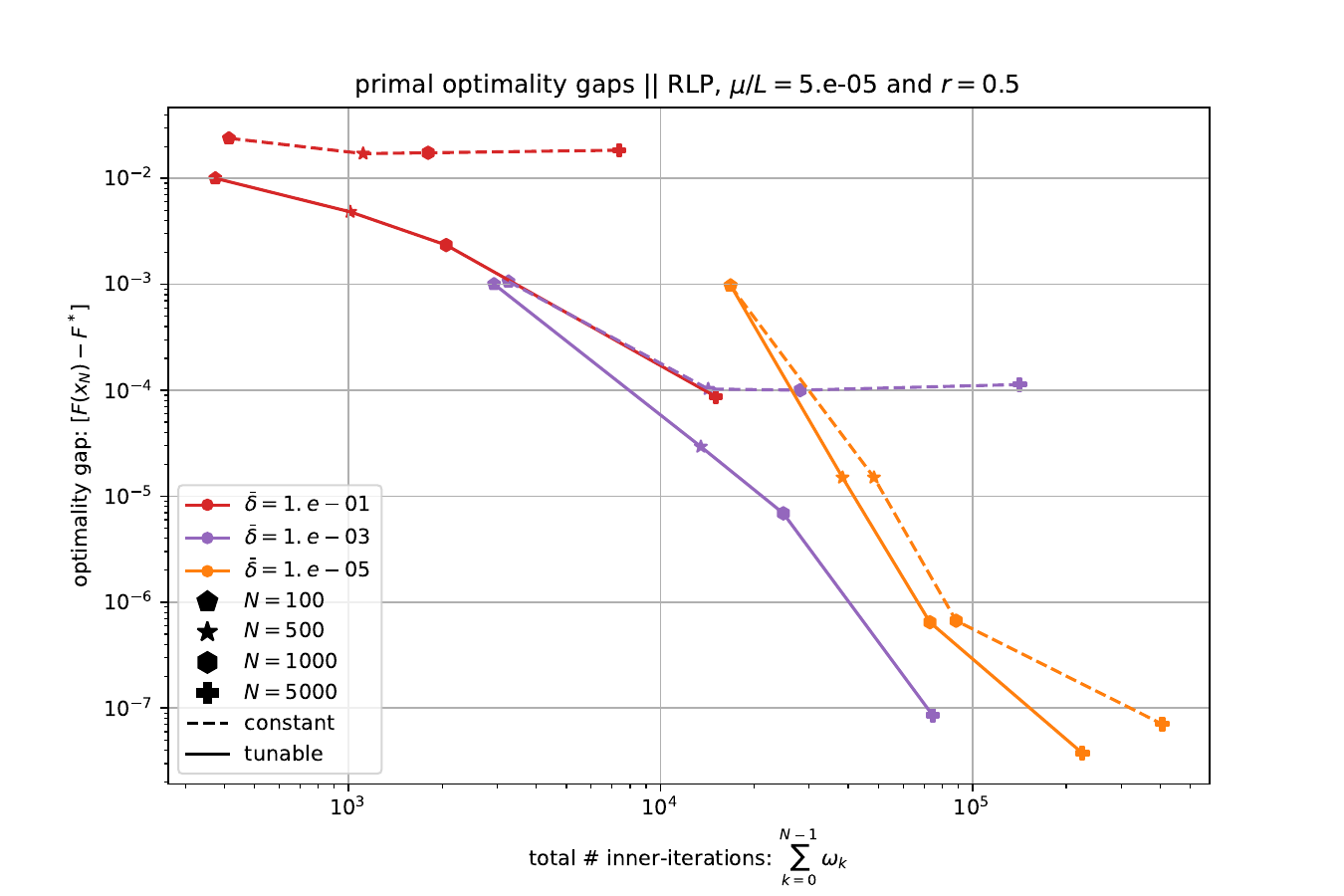} }}%
    \caption{(regularization, $\mu>0$) | R.O. over Convex Hull, \emph{offline} with multiple $\bar{\delta}$}%
    \label{fig:rlp_reg}%
\end{figure}

\newpage
\subsubsection*{Online Schedules (Experiment 3)}
\paragraph*{Test instances}
Finally, we run FGM with adaptive stepsizes $L_k^{-1} > 0$ for any $k \geq 0$. A stepsize is validated at iteration $k$ as soon as the inequality 
\begin{equation}
\tilde{f}(x_{k+1}) \leq \tilde{f}(y_k) + \langle \nabla \tilde{f}(y_k) , x_{k+1}-y_k \rangle + \frac{L_{k+1}}{2}||x_{k+1}-y_k||^2 + 2\,\delta_k
\label{eq:local_smoothness_validation}
\end{equation}
holds for $(2\delta_k, \sigma^{-1}+\mu,\mu)$ inexact information tuples at feasible points $x_{k+1}$ and $y_k$ respectively. FGM increases by $3/2$ the stepsize when an iteration succeeds, i.e.\@ \eqref{eq:local_smoothness_validation} passes, and decreases it by a factor $2$ otherwise. Note that \eqref{eq:local_smoothness_validation} is always satisfied for $L_{k+1} \geq \sigma^{-1} + \mu$. Convergence guarantees like \eqref{eq:convergence_model_example_2} are conserved \cite{Stonyakin20} but now the \emph{impact coefficients} $a_k = A_{k+1}$ depend on $L_{k+1}$ via the recursion ($A_0 = 0$)
\begin{equation}
A_{k+1}(1+\mu A_k) = L_{k+1}(A_{k+1}-A_k)^2
\label{eq:recursion_impact}
\end{equation}
and are therefore not predictable in advance. This paves the way for the application of
our \emph{online} approach \eqref{eq:delta_optimal_simplified_contd} with $N_r=50$. $N_r$ is kept small in practice, its sole interest being to specify the very first inexactness parameters before the \emph{online} approach takes over. Yet, one must simulate the $N_r$ first coefficients $\{a_k\}_{k=0}^{N_r-1}$ required to devise the \emph{offline} optimized schedule $\{\delta_k^*\}_{k=0}^{N_r-1}$ from \eqref{eq:master_problem}. To that purpose, one fakes the value of $L_{k+1}$ fixing it to $L=\sigma^{-1}+\mu$ within \eqref{eq:recursion_impact} and uses the computed $\{a_k = A_{k+1}\}_{k=0}^{N_r-1}$. \\
For each level of \emph{reference} inexactness $\bar{\delta} \in \{9\cdot10^{-3},10^{-4},10^{-6}\}$, we tried 4 different inexactness schedules. I.e. for any iteration index $k\geq0$, the \emph{constant} one $\delta_k=\bar{\delta}$, the heuristic \emph{online tunable} (as described above) with $a_k = A_{k+1}$, the fully sublinear monotonically decreasing tuned as in \cite{Villa13} (\texttt{poly-3}), $\delta_k  = \bar{\delta}\,(k+1)^{-3}$ and the fully linear one \cite{BTB22} (\texttt{linear}) with \@$\delta_k  = \bar{\delta}\,(1-\sqrt{\mu / L })^{-k}$.
For this experiment, $\hat{x}^*$ (see \eqref{eq:lower_bound_opt}) is taken as the output of adaptive FGM after at $3 \cdot 10^3$ iterations with \emph{constant} inexactness of $10^{-12}$ for $r=0$ and $10^{-8}$ otherwise for $r=\frac{1}{2}$. 
We have conducted $5$ random initializations of the settings;
$(d,n,p,\sigma) =(50,100,5^{-1},10^{-3})$ which, akin to Experiment 2, led to $\hat{\kappa}=0$ and thus $r=\frac{1}{2}$ and $(d,n,p,\sigma) =(200,100,5^{-1},10^{-3})$ that gave $\hat{\kappa}>0$ and, accordingly, $r=0$. Finally, we set the regularization to $\mu=10^{-1}$.\newpage

\paragraph*{Results} We decided to keep track of the primal objective value every $10$ outer-iteration, i.e.\@ when $k\, \text{mod}\, 10 = 0$, we did compute $\tilde{f}(x_k) \simeq f(x_k)$ with precision $10^{-10}$. After the runs, we turned these primal objective values into primal optimality gaps thanks to our estimation of $F^*$.  
We can observe on Figure \ref{fig:online1}, \ref{fig:online2} and \ref{fig:online3}  that our \emph{online} heuristic \emph{tunable} inexactness schedule adapts well to the \emph{error-free} speed of convergence of FGM, symbolized by the value of the main sequence $\{A_{k}\}_{k \in \mathbb{N}}$. \\We recall that its starting growth is of order $\mathcal{O}(k^2)$ but eventually, since $\mu >0$ in the present experiment, a linear growth rate shows up lately \cite{Stonyakin20,BTB22}. It is worth noticing that (\texttt{linear}) schedule was the best in the low \emph{reference} inexactness regime. Nevertheless, overall, neither (\texttt{poly-3}) or (\texttt{linear}) adapts as well as our schedule. Indeed, although comparable to \emph{tunable} when $r = \frac{1}{2}$ on Figure \ref{fig:online1} (b), \ref{fig:online2} (b) and \ref{fig:online3} (b), \texttt{poly-3} does not take advantage of the regularization whereas \texttt{linear} appears way too conservative on Figure \ref{fig:online2} and \ref{fig:online3}, perhaps in order to preserve as much as possible the asymptotical rate of convergence of the \emph{error-free} counterpart of FGM. The reference \emph{constant} schedule for inexactness turns out to be competitive until \emph{error-accumulation} (see Remark \ref{inexact_model_comments2}) undermines further improvements.
\begin{figure}[ht]%
    \centering
    \subfloat[\centering smooth strongly convex inner problem ($r=0$)]{{\includegraphics[width=7.3cm]{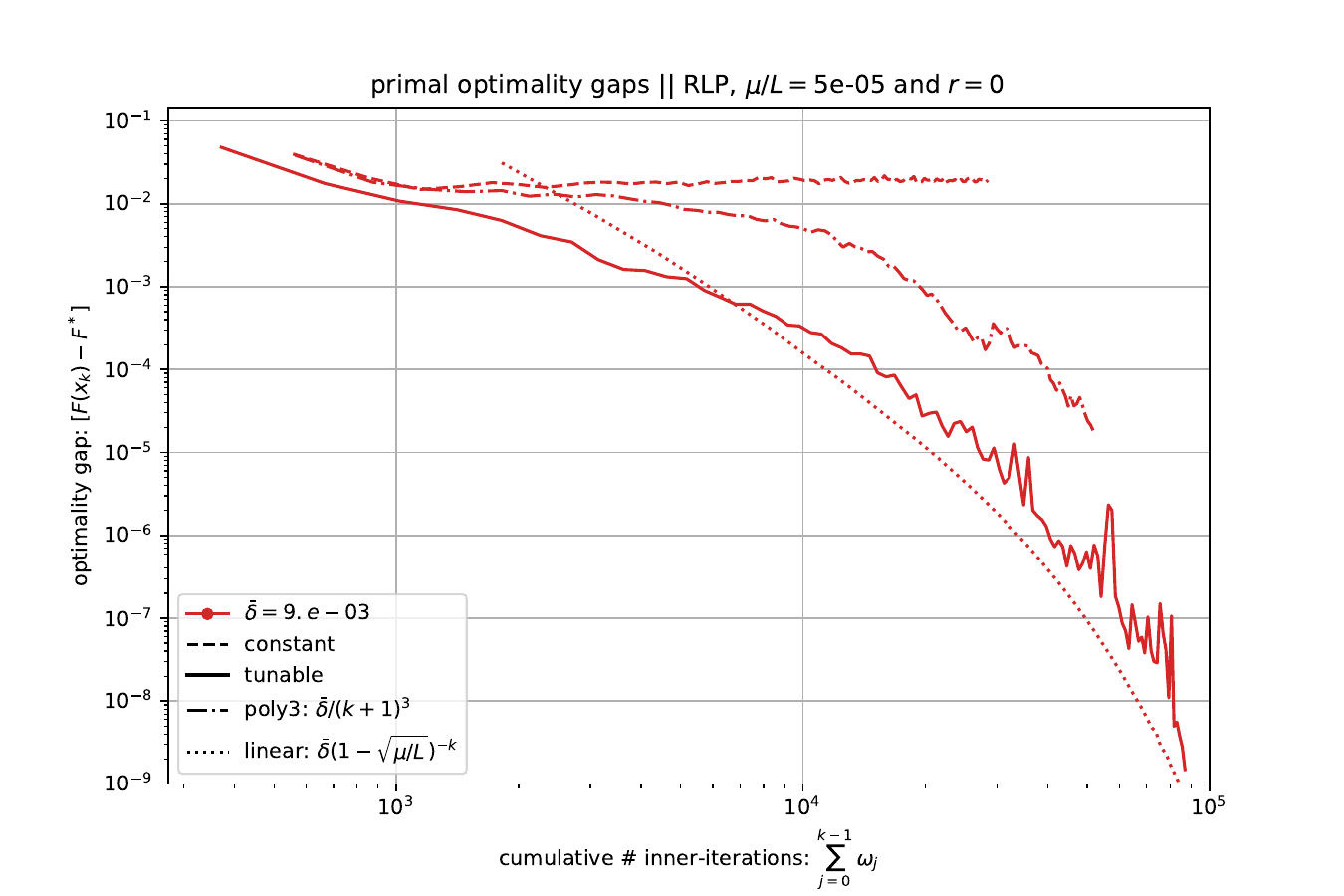} }}%
    \subfloat[\centering smooth convex inner problem ($r=\frac{1}{2}$)]{{\includegraphics[width=7.3cm]{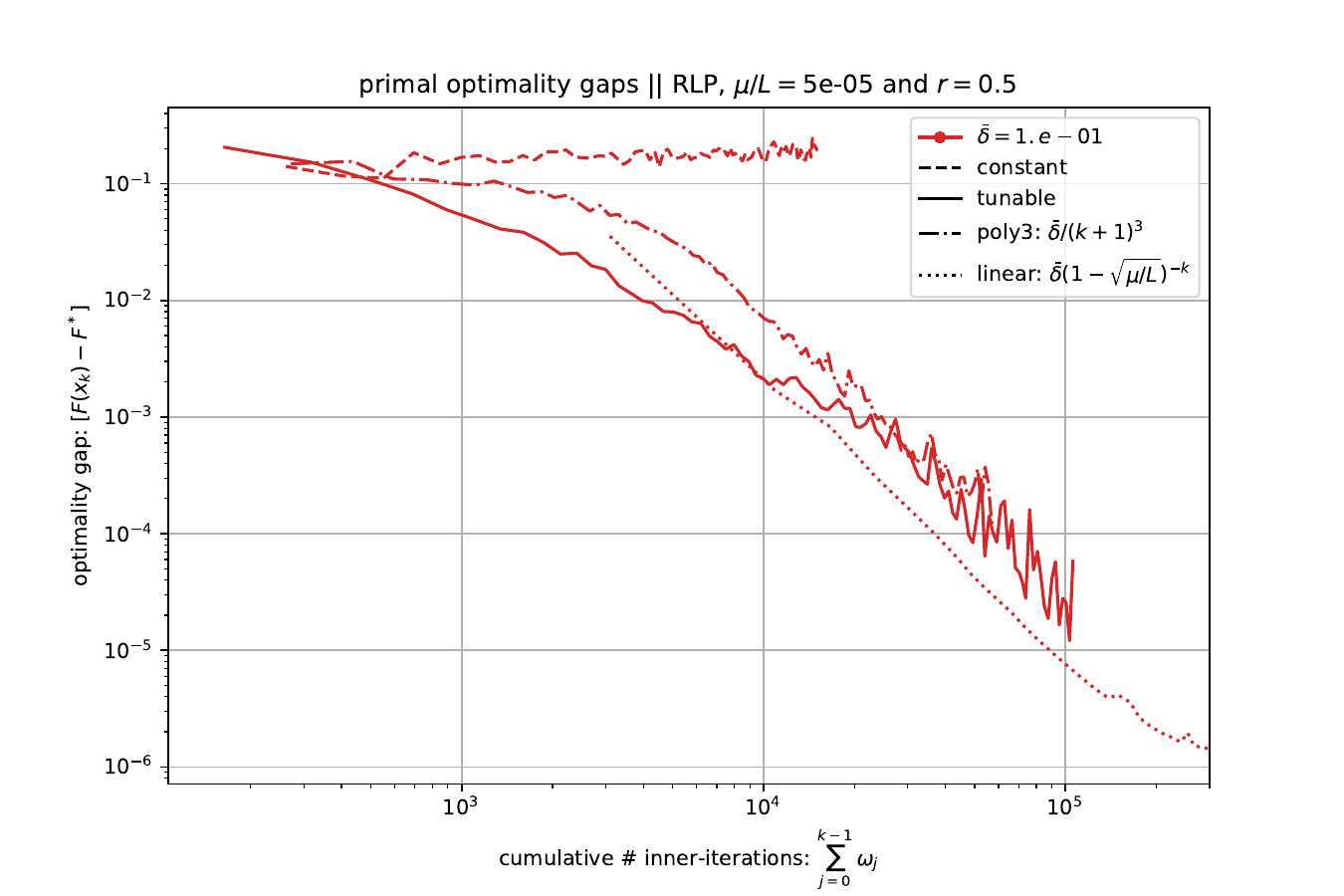} }}%
    \caption{(regularization, $\mu>0$) | R.O. over Convex Hull, \emph{online} with low accuracy $\bar{\delta}$ }%
    \label{fig:online1}%
\end{figure}
\begin{figure}[ht]%
    \centering
    \vspace{-20pt}
    \subfloat[\centering smooth strongly convex inner problem ($r=0$)]{{\includegraphics[width=7.3cm]{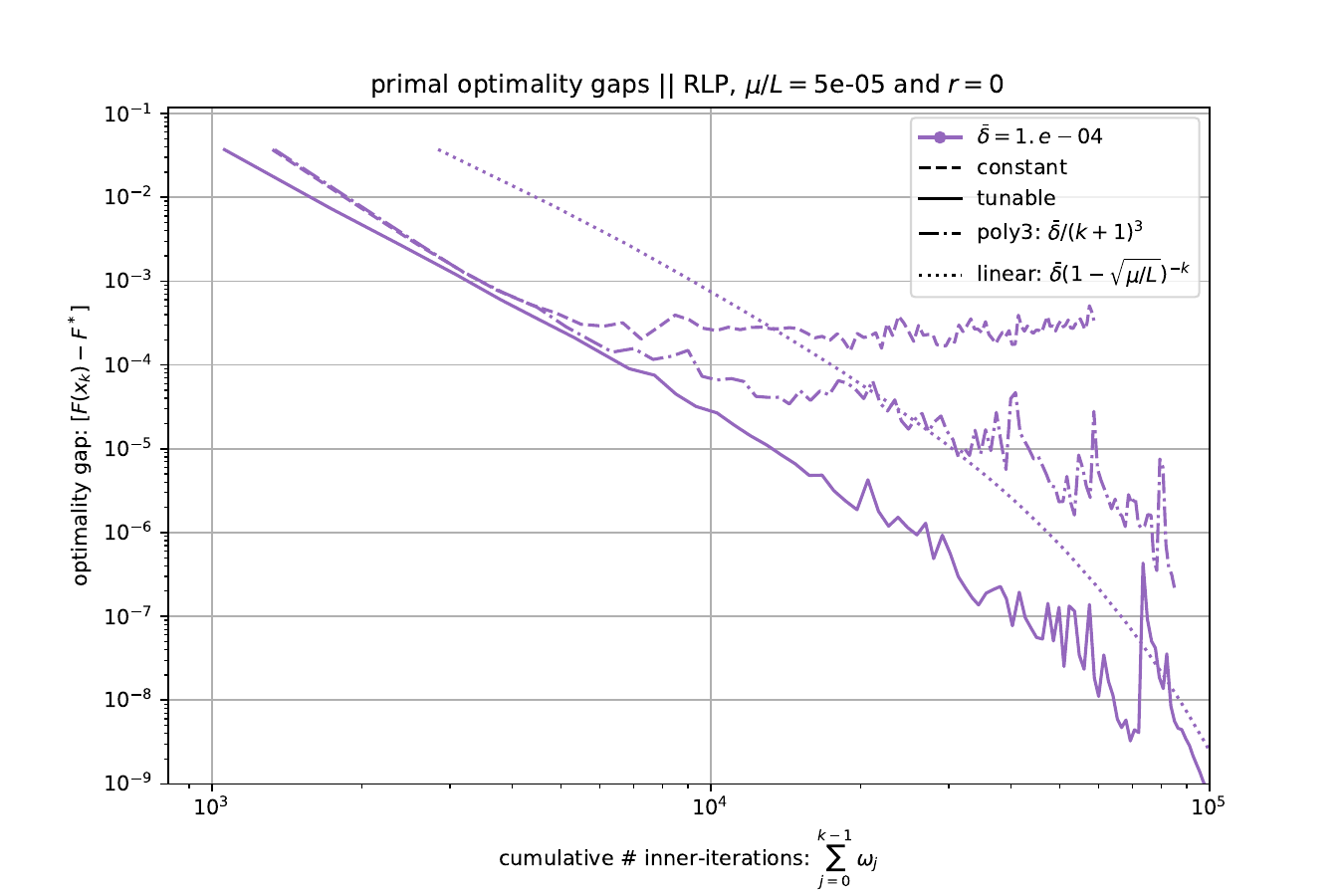} }}%
    \subfloat[\centering smooth convex inner problem ($r=\frac{1}{2}$)]{{\includegraphics[width=7.3cm]{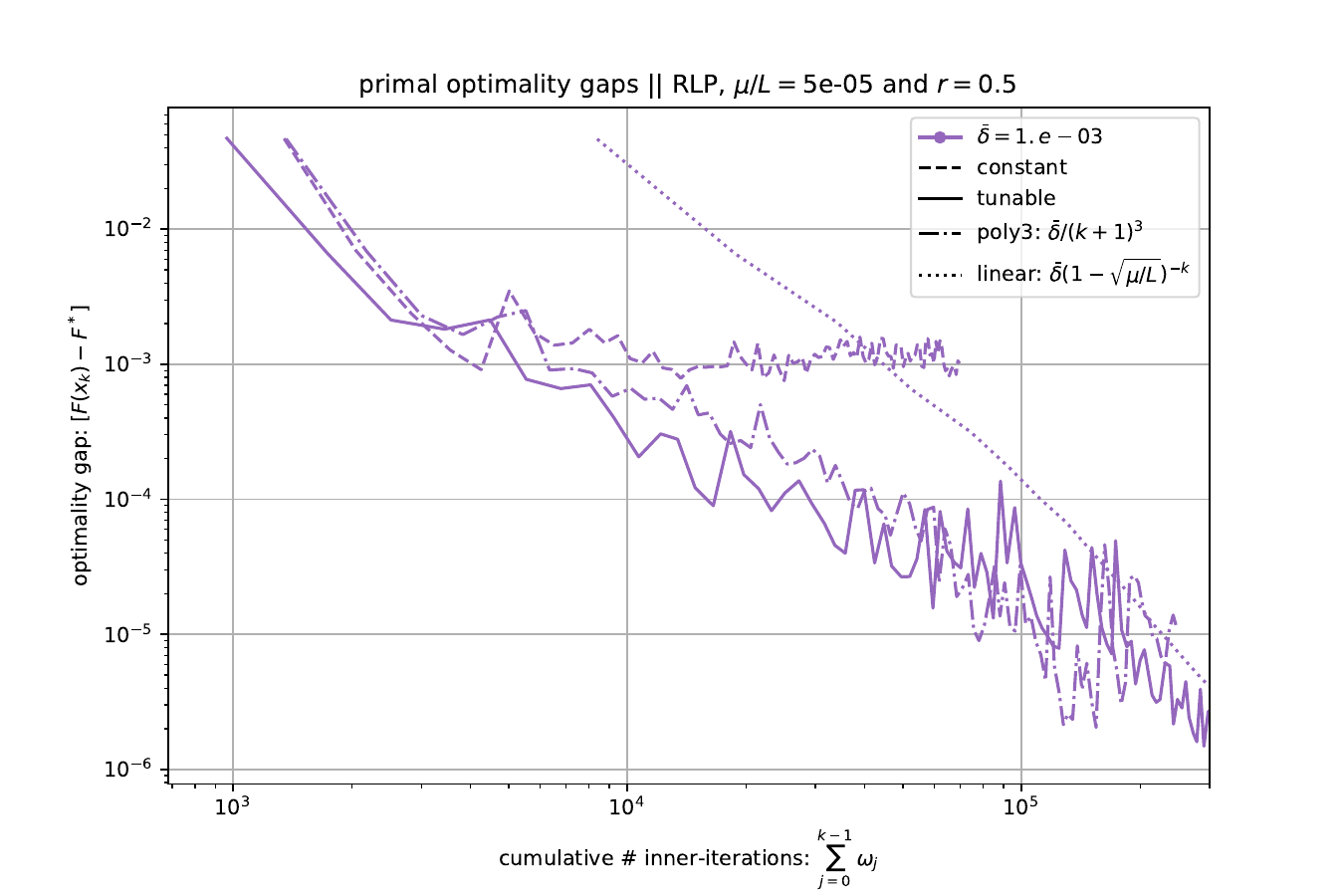} }}%
    \caption{(regularization, $\mu>0$) | R.O. over Convex Hull, \emph{online} with medium accuracy $\bar{\delta}$ }%
    \label{fig:online2}%
\end{figure}
\begin{figure}[ht]%
    \centering
   \vspace{-20pt}
    \subfloat[\centering smooth strongly-convex inner problem ($r=0$)]{{\includegraphics[width=7.3cm]{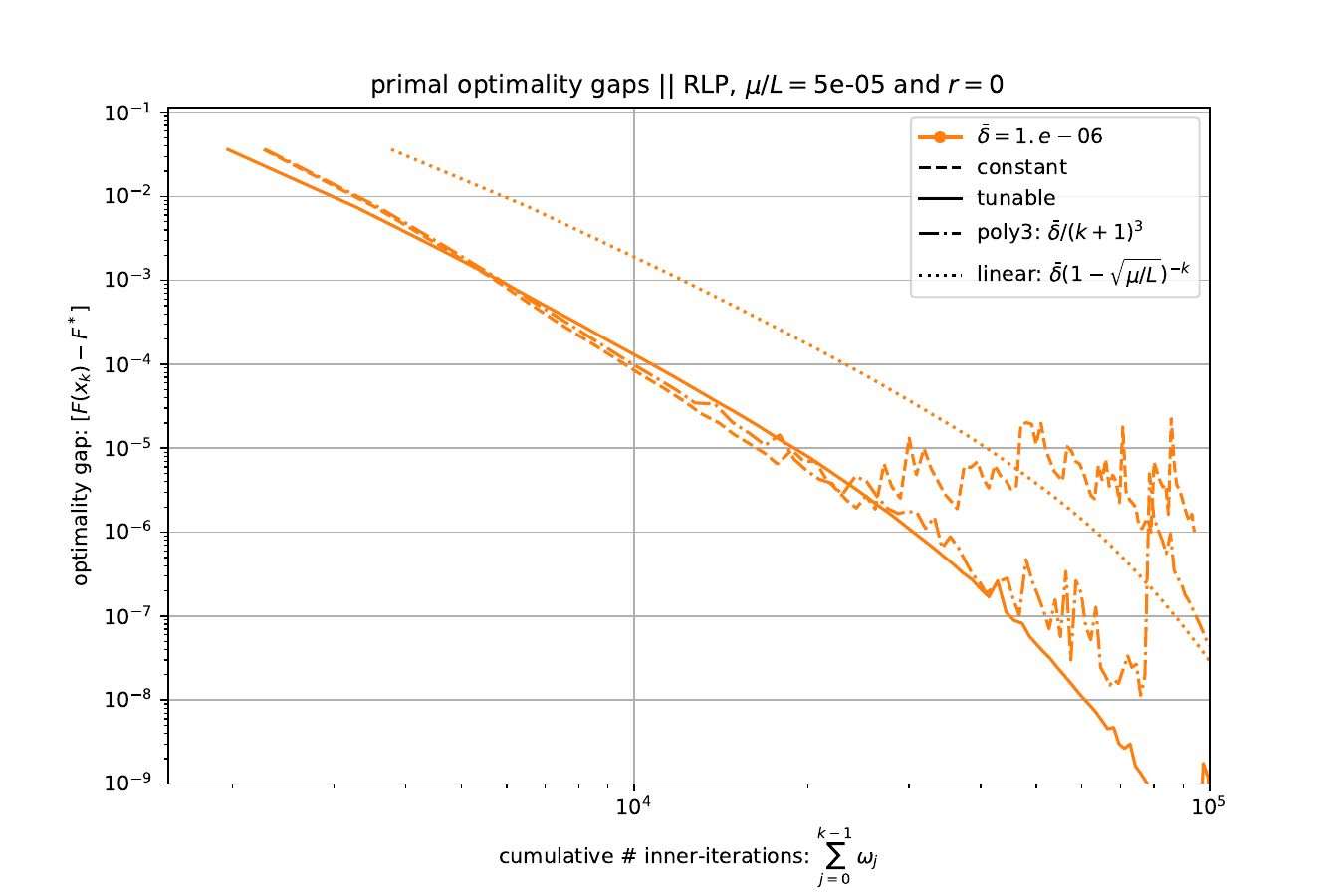} }}%
    \subfloat[\centering smooth convex inner problem ($r=\frac{1}{2}$)]{{\includegraphics[width=7.3cm]{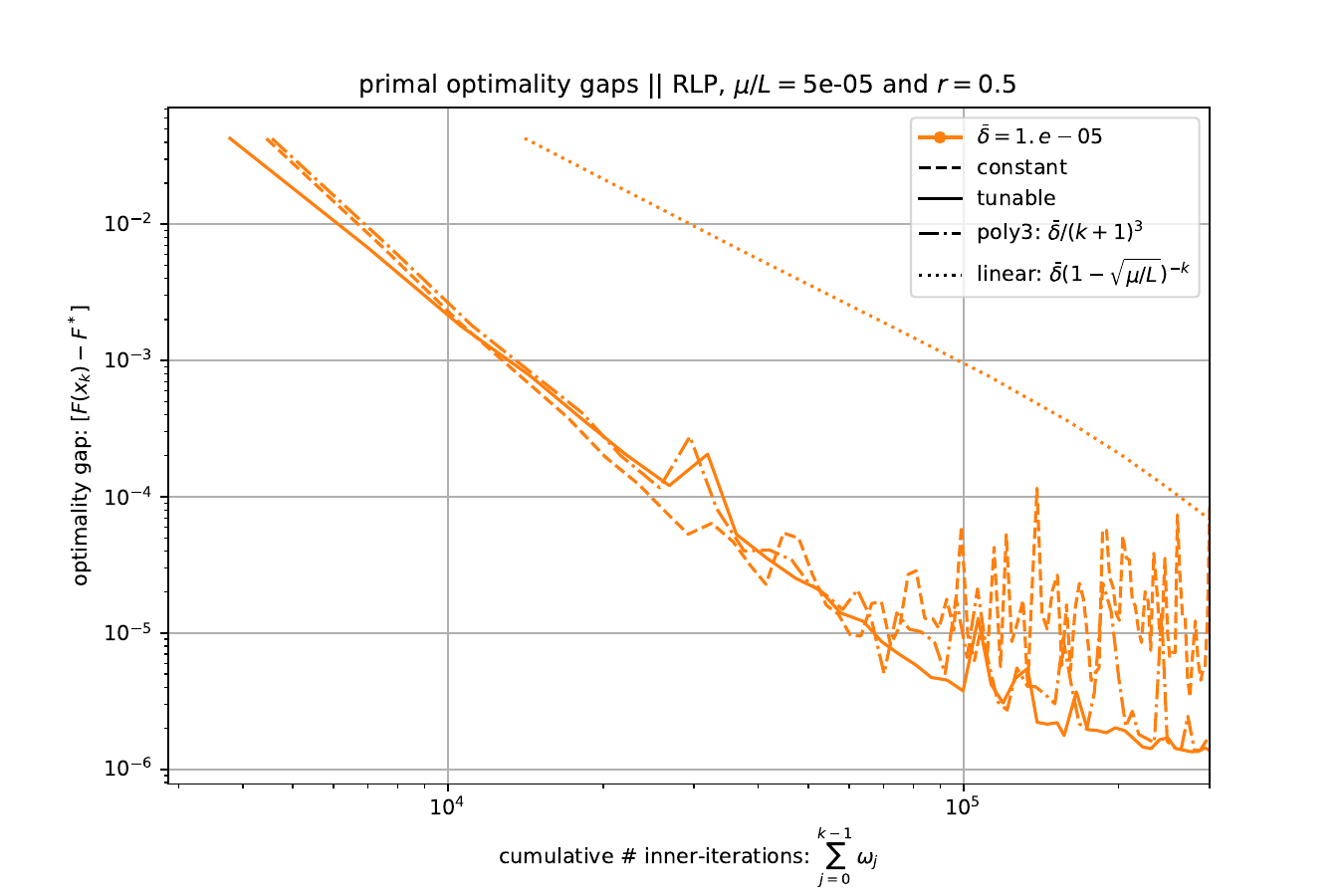} }}%
    \caption{(regularization, $\mu>0$) | R.O. over Convex Hull, \emph{online} with high accuracy $\bar{\delta}$ }%
    \label{fig:online3}%
\end{figure}
}

\newpage
\section{Conclusion}
In this paper we considered a class of iterative algorithms, namely Tunable Oracle Methods (TOM), for which one can take advantage from the combined knowledge of both the computational cost associated with the oracle calls and the impact of associated inexactness on the convergence. We have shown how to choose optimally the level of inexactness that should requested at each iteration. Our numerical experiments confirm the superiority of these optimal schedules over the use of constant inexactness, for a given total computational budget, and also compare favorably with existing baselines from the literature. 
\edited{
Future work may include a \emph{tight analysis} of more iterative methods that involve controllable inexactness, e.g.\@ bilevel learning, for which we can also hope that optimal inexactness schedules  enhance the practical performances.}
\edited{
Another direction for future research would be to gain insight about \emph{random inexactness} rather than the (possibly) adversarial model considered in this paper. In this context, the goal could be to choose the appropriate amount of work to drive the distribution of inexactness towards the local needs of the iterative algorithm.}

\section*{Acknowledgement(s)}

The authors are grateful to Yurii Nesterov and Pierre-Antoine Absil for their advice both in terms of content and presentation of the results.

\section*{Disclosure statement}
No potential conflict of interest was reported by the author(s).
\section*{Funding}

Guillaume Van Dessel is funded by the UCLouvain university as a teaching assistant. 
\bibliographystyle{abbrv}
\bibliography{TOM.bib}
\newpage 
\appendix
\addcontentsline{toc}{section}{Appendices}

\section{Miscellaneous proofs}
\label{proof_theorems}
\subsection{Proof of Theorem \ref{consistency}.} \label{proof_theorem1}
\begin{proof} We proceed in two sequential steps: proving \eqref{eq:consistency} then \eqref{eq:consistency2} with the help of \eqref{eq:consistency}.  Let us proceed by \textit{absurdum} by assuming that there exists a pair of indices \\$(k_{1},k_{2}) \in \{0,\dots,N-1\}^2$ such that $$\mathcal{I}_{k_{1}} \succeq \mathcal{I}_{k_{2}}\hspace{3pt} \wedge \hspace{3pt} \sigma_{k_{1}}^* < \sigma_{k_{2}}^*$$ We consider the function $$(\sigma_{k_{1}},\sigma_{k_{2}}) \in \mathbb{R}^2 \to f(\sigma_{k_{1}},\sigma_{k_{2}}) := n_{k_{1}}(\sigma_{k_{1}}) + n_{k_{2}}(\sigma_{k_{2}}) + \sum_{k=0,\,k\not=k_1,\,k\not=k_2}^{N-1}\,n_k(\sigma_k^*) \in \mathbb{R}_{+}$$

Trivially, the entries $\sigma_{k_{1}}^*$ and $\sigma_{k_{2}}^*$ introduced before form a solution of the two dimensional problem 

$$\min_{\sigma_{k_{1}}\in \Xi,\hspace{2pt}\sigma_{k_{2}}\in \Xi}\,f(\sigma_{k_{1}},\sigma_{k_{2}}) \hspace{5pt}\text{s.t.}\hspace{5pt} d_{k_1}(\sigma_{k_{1}}) + d_{k_2}(\sigma_{k_{2}}) = \tilde{D}= D-\sum_{k=0,\,k\not=k_1,\,k\not=k_2}^{N-1}\,d_{k}(\sigma_k^*)$$
 
Let us show that $\sigma_{k_{1}}^* < \sigma_{k_{2}}^*$ is impossible under the hypotheses above. KKT necessary conditions ensure the existence of $\tilde{\lambda}\in\mathbb{R},\,\tilde{\mu}_{k_{1},\oplus},\,\tilde{\mu}_{k_{1},\ominus} \geq 0$ and $\tilde{\mu}_{k_{2},\oplus},\,\tilde{\mu}_{k_{2},\ominus} \geq 0$ with

\begin{align*}
n'_{k_1}(\sigma_{k_{1}}^*) + \tilde{\lambda}\,d'_{k_1}(\sigma_{k_{1}}^*) &= (\tilde{\mu}_{k_{1},\ominus}-\tilde{\mu}_{k_{1},\oplus})\\
n'_{k_2}(\sigma_{k_{2}}^*) + \tilde{\lambda}\,d'_{k_2}(\sigma_{k_{2}}^*) &= (\tilde{\mu}_{k_{2},\ominus}-\tilde{\mu}_{k_{2},\oplus})\\
(l-\sigma_{k_{1}}^*)\cdot \tilde{\mu}_{k_{1},\ominus} &= 0\\
(l-\sigma_{k_{2}}^*)\cdot \tilde{\mu}_{k_{2},\ominus} &= 0\\
(u-\sigma_{k_{1}}^*)\cdot \tilde{\mu}_{k_{1},\oplus} &= 0\\
(u-\sigma_{k_{2}}^*)\cdot \tilde{\mu}_{k_{2},\oplus} &= 0
\end{align*}

The only possible scenarios leading to $l \leq \sigma_{k_{1}}^* < \sigma_{k_{2}}^*\leq u$ are achieved when $\tilde{\mu}_{k_{1},\oplus} = \tilde{\mu}_{k_{2},\ominus} = 0$ implying that \begin{equation} 
- \frac{n'_{k_2}(\sigma_{k_{2}}^*) + \tilde{\mu}_{k_2,\oplus}}{d'_{k_2}(\sigma_{k_2}^*)} = \tilde{\lambda} = - \frac{n'_{k_1}(\sigma_{k_{1}}^*) - \tilde{\mu}_{k_1,\ominus}}{d'_{k_1}(\sigma_{k_1}^*)}
\label{eq:main_lambda_tilde}
\end{equation}
After elementary calculations, it must follow that $$ \frac{1+\frac{\tilde{\mu}_{k_{2},\oplus}}{n'_{k_{2}}(\sigma_{k_{2}}^*)}}{1-\frac{\tilde{\mu}_{k_{1},\ominus}}{n'_{k_{1}}(\sigma_{k_{1}}^*)}}  = \frac{\mathcal{I}_{k_{1}}(\sigma_{k_1}^*)}{\mathcal{I}_{k_{2}}(\sigma_{k_2}^*)} < 1$$

The inequality comes from the fact that $\{\mathcal{I}_k\}_{k=0}^{N-1}$ functions are strictly \edited{increasing} and strictly negative on $\Xi$ while $\sigma_{k_{1}}^* < \sigma_{k_{2}}^*$. \edited{Indeed, our inductive step leads to $$\mathcal{I}_{k_{1}}(\sigma_{k_{1}}^*) -  \mathcal{I}_{k_{2}}(\sigma_{k_{1}}^*) \geq 0 \Rightarrow 0 < \mathcal{I}_{k_{1}}(\sigma_{k_{1}}^*)/\mathcal{I}_{k_{2}}(\sigma_{k_{1}}^*) \leq 1$$ because of the strict negativity of $\{\mathcal{I}_k\}_{k=0}^{N-1}$. One then notices $\mathcal{I}_{k_2}(\sigma_{k_1}^*) < \mathcal{I}_{k_2}(\sigma_{k_2}^*)$.} \\\edited{Therefore}, one should observe that 
    $$ 0 \leq 1+\frac{\tilde{\mu}_{k_{2},\oplus}}{n'_{k_{2}}(\sigma_{k_{2}}^*)} < 1-\frac{\tilde{\mu}_{k_{1},\ominus}}{n'_{k_{1}}(\sigma_{k_{1}}^*)}$$
which turns out to be impossible since $\{n_k\}_{k=0}^{N-1}$ functions are \edited{strictly} increasing on $\Xi$. \\We conclude that \eqref{eq:consistency} is correct. To prove \eqref{eq:consistency2} on the other hand, let us imagine this time that there exists a pair of indices $(k_{1},k_{2}) \in \{0,\dots,N-1\}^2$ such that $$\mathcal{I}_{k_{1}} \succ \mathcal{I}_{k_{2}}\hspace{3pt} \wedge \hspace{3pt} l< \sigma_{k_{1}}^* \leq \sigma_{k_{2}}^* < u$$
Clearly, \eqref{eq:consistency} already shows that $\sigma_{k_{1}}^* \not< \sigma_{k_{2}}^*$. The only plausible scenario left resumes in $l< \sigma_{k_{1}}^* = \sigma^* = \sigma_{k_{2}}^* < u$. Setting $\tilde{\mu}_{k_{1},\oplus} = 0 = \tilde{\mu}_{k_{2},\ominus}$ in \eqref{eq:main_lambda_tilde} yields $$- \mathcal{I}_{k_{2}}(\sigma^*) = \tilde{\lambda} = - \mathcal{I}_{k_{2}}(\sigma^*)$$
Again, this contradicts our assumption that $\mathcal{I}_{k_{1}} \succ \mathcal{I}_{k_{2}}$.
\end{proof}

\subsection{Proof of Theorem \ref{general_theorem}.}
\label{proof_theorem2}
\begin{proof}
A neutral schedule $\{\delta_k = \bar{\delta}\}_{k=0}^{N-1}$ certifies, trivially, the feasibility of problem \eqref{eq:master_problem} whose basic domain $\Xi^N$ is compact.
Since $\bar{\delta} \in \text{int}(\Xi) = ]m \bar{\delta},\, M \bar{\delta}[$ is fixed, according to the definition of $h$ within Assumption B, we can write $$0 < \sum_{k=0}^{N-1}\,b_k\,h(\bar{\delta}) = \bar{\mathcal{B}} < \infty$$

Before diving into KKT conditions, let us point out two observations:
\begin{enumerate}
    \item With $m<1$, $\{\delta_k = m \bar{\delta}\}_{k=0}^{N-1}$ is clearly not feasible since $h$ is assumed to be strictly decreasing on $\Xi$. 
    \item With $M>1$, $\{\delta_k = M \bar{\delta}\}_{k=0}^{N-1}$ is clearly not optimal since the neutral schedule above is feasible and  
    $$ \sum_{k=0}^{N-1}\,a_k\,\bar{\delta} = ||a||_1\,\bar{\delta} < M\,||a||_1\,\bar{\delta} = \sum_{k=0}^{N-1}\,a_k\,M \bar{\delta}$$ 
    as $a \succ \mathbf{0}$ according to Assumption \eqref{assumption_inexact_model}.
\end{enumerate} 

Assumptions of Theorem \ref{general_theorem} satisfy the ones required by Theorem \ref{consistency} so that this latter applies verbatim to problem \eqref{eq:master_problem} with $n_k(\cdot) = a_k\,(\cdot)$ and $d_k(\cdot) = b_k\, h(\cdot)$. \\The negative sign of $h'$ ensures the soundness of our ranking (cfr. \eqref{eq:ordering_general}). $$\rho(k_{1}) < \rho(k_{2}) \Rightarrow \nu_{k_{1}} = b_{k_{1}} / a_{k_{1}} \geq b_{k_{2}} / a_{k_{2}}\ = \nu_{k_{2}}$$ leading to $$ \mathcal{I}_{k_{1}}(\cdot) = a_{k_{1}} / (b_{k_{1}}\,h'(\cdot)) \succeq a_{k_{2}} / (b_{k_{2}}\,h'(\cdot))=\mathcal{I}_{k_{2}}(\cdot) $$ Since $n'_k / d'_k$ ratios are strictly \edited{increasing} and negative, we can claim that there exists  $N_{\oplus},N_{\ominus} \in \{0,\dots,N-1\}$ defining two subsets $\oplus = \{k \,|\,\rho(k) < N_{\oplus}\}, \ominus=\{k \,|\,\rho(k) > N-1- N_{\ominus}\}$ of $\{0,\dots,N-1\}$ such that $\delta_k^*$ will be affected to the upper-bound $u = M\bar{\delta}$ if $k \in \oplus$ whereas it will be assigned the lower-bound $l = m\bar{\delta}$ if $k\in \ominus$.
By means of our two previous observations, $$N_{\oplus},N_{\ominus} <N$$ 

If $(\oplus,\ominus)$ forms a partition of $\{0,\dots,N-1\}$  with
    \begin{equation}
    \sum_{k \in \oplus}\,b_k\,h(M\bar{\delta}) + \sum_{k \in \ominus}\,b_k\,h(m\bar{\delta}) = \bar{\mathcal{B}}
    \label{eq:degenerate}
    \end{equation}
    then the inherent schedule $\delta_k = M\bar{\delta}$ for every $k \in \oplus$, $\delta_k = m\bar{\delta}$ for every $k \in \ominus$ might be a solution of \eqref{eq:master_problem}. The structure of such schedule is identical to the one expected by \eqref{eq:general_optimal_schedule}. We consider now other possible solutions.
KKT necessary conditions \edited{(see LICQ below)} imply that if $\delta^*$ is a local optimum of \eqref{eq:master_problem}, then there exist $\tilde{\lambda} \in \mathbb{R}$, $\tilde{\mu}_{\ominus},\,\tilde{\mu}_{\oplus} \in \mathbb{R}_+^{N}$ (dual feasibility) such that equalities and equalities of Table \ref{tab:KKT} are jointly satisfied. When the context is not misleading, we make the abuse of notation for any $\delta \in \Xi^N$, $h'(\delta) = (h'(\delta_0),\dots,h'(\delta_{N-1}))^T$ and $h(\delta) = (h(\delta_0),\dots,h(\delta_{N-1}))^T$.
\renewcommand{\arraystretch}{1.5}
\begin{table}[ht]
\centering
\caption{KKT conditions for \eqref{eq:master_problem}}
\begin{tabular}{|c|c|}
\hline
    $a + \tilde{\lambda}\,b \odot \,h'(\delta^*)  = (\tilde{\mu}_{\ominus}-\tilde{\mu}_{\oplus})$ & stationarity\\
    $m \bar{\delta} \,\preceq \, \delta^* \, \preceq \,M \bar{\delta}$ & primal feasibility (bounds)\\
    $b^T\,h(\delta^*) = \bar{\mathcal{B}}$ & primal feasibility (budget)\\
    $(m \bar{\delta} -\delta^*) \odot \tilde{\mu}_{\ominus} = \mathbf{0}$ & complementary (-) \\
    $(\delta^* - M \bar{\delta})\odot \tilde{\mu}_{\oplus} = \mathbf{0}$ & complementary (+)\\
    \hline
\end{tabular}
\label{tab:KKT}
\end{table}
\vspace{-7pt}
\paragraph*{LICQ} Indeed, \eqref{eq:master_problem} enjoys the LICQ constraint qualification at any $\delta \in \Xi^N$ not matching the degenerate case \eqref{eq:degenerate}. At any feasible $\delta \in \Xi^N$, the span of active constraints' gradients always involve $h'(\delta)$ whose entries are all non-zero, excepted eventually at $k_{+} \in \oplus$ where $\delta_{k_{+}} = M \bar{\delta}$ if $h'(M\bar{\delta}) = 0$ (due to the convexity of $h$ and strict monotony of $h'$). Gradients of other active constraints are vectors $\mathbf{e}_{k_{+}}$ for $k_{+} \in \oplus$ and $\mathbf{e}_{k_{-}}$ for $k_{-} \in \ominus$ that are all linearly independent. Since we explicitly discard schedules of \eqref{eq:degenerate} type, it must exist an entry $k \in \{0,\dots,N-1\}$ where bound constraints are inactive, hence $h'(\delta_k) < 0$. In order to obtain the $\mathbf{0}$ vector from a linear combination of active constraints' gradients, one should then take zero times $h'(\delta)$. If no bound constraint is active at $\delta$ then the induced linear combination immediately turns out to be the trivial one. Conversely, if some bound constraints are active then their associated coefficients in the linear combination must also be zero as shown previously, giving rise again to the null combination. \vspace{5pt}

Note that KKT conditions are sufficient since the relaxation (convex) of \eqref{eq:master_problem}
\begin{equation}
\min_{m \bar{\delta} \preceq \delta \preceq M \bar{\delta}}\,\sum_{k=0}^{N-1}\,a_k\,\delta_k \hspace{3pt}\text{s.t.}\hspace{3pt} \sum_{k=0}^{N-1}\,b_k\,h(\delta_k) \leq \bar{\mathcal{B}}\end{equation}
would always lead to an optimal solution that tightens the budget constraint as $a \succ \mathbf{0}$ and $h$ is strictly decreasing. For writing ease, let us now denote $$\mathcal{T} = \{0,\dots,N-1\} \backslash (\oplus \cup \ominus)$$ In line with the comments made so far, we look after a schedule $\{\delta_k^*\}_{k=0}^{N-1}$ of the form
\begin{equation}
\begin{cases} \delta_k^* =  M\,\bar{\delta} & k\in \oplus \\ \delta_k^* \,\in\, ]m\bar{\delta},M\bar{\delta}[ & k\in \mathcal{T} \\ \delta_k^* = m\,\bar{\delta} & k \in \ominus \end{cases}
\label{eq:pre_general_optimal_schedule}
\end{equation}
for two positive integers $N_{\oplus},N_{\ominus}$ with $0 \leq N_{\oplus}+N_{\ominus}\leq N-1$. \\
\vspace{2pt}
 \hspace{-5pt}For any $k \in \mathcal{T}$, stationarity requirement compels \vspace{-10pt}$$ a_k + \tilde{\lambda}\,b_k\,h'(\delta_k^*) = 0 \Rightarrow h'(\delta_k^*) = -(\tilde{\lambda}\,a_k^{-1} \,b_k)^{-1} = \frac{a_k\,\lambda^*}{b_k} \Rightarrow \delta^*_k = (h')^{(-1)}\bigg(\frac{a_k\,\lambda^*}{b_k}\bigg) $$for $\lambda^* = -(\tilde{\lambda})^{-1} \in \mathbb{R}$. Primal feasibility (budget) row of Table \ref{tab:KKT} translates to \eqref{eq:budget_general}.\end{proof}

\subsection{Uniqueness.} 
\label{uniqueness_explanation}
\begin{proof}
We prove here the uniqueness of $\delta^*$. For ease of writing, let us reuse the notation $$\sum_{k=0}^{N-1}\,b_k\,h(\bar{\delta}) = \bar{\mathcal{B}}$$
Aside from the fact that all the entries $\nu_k$ of $\nu$ supposedly differ, Assumption \ref{assumption_cost_model} about $h$ certifies that for any fixed $N_{\oplus}$, $N_{\ominus}$ , there can be at most one $\lambda^* \in \mathbb{R}$ ensuring \eqref{eq:budget_general}. Seen from a different angle, for any couple $(N_{\oplus},N_{\ominus})$ there can be a single $\{\delta_k^*(N_{\oplus},N_{\ominus})\}_{k=0}^{N-1}$ respecting both \eqref{eq:general_optimal_schedule} and \eqref{eq:budget_general}. Let us assume by \emph{absurdum} the optimality of two such schedules. Then, there must exist two couples $(N_{\oplus,1},N_{\ominus,1})$ and $(N_{\oplus,2},N_{\ominus,2})$ such that $$ \sum_{k=0}^{N-1}\,a_k\,\delta_k^*(N_{\oplus,1},N_{\ominus,1}) = u ^* = \sum_{k=0}^{N-1}\,a_k\,\delta_k^*(N_{\oplus,2},N_{\ominus,2})$$
Any schedule $\hat{\delta}^{(t)}$ \emph{in between} for $t \in ]0,1[$ with $$\hat{\delta}^{(t)}_k = t\cdot \delta_k^*(N_{\oplus,1},N_{\ominus,1}) + (1-t)\cdot \delta_k^*(N_{\oplus,2},N_{\ominus,2})$$ for every $k \in \{0,\dots,N-1\}$ would trivially also yield the optimal cost $u^*$. By convexity of the set $\mathcal{H} \subset \Xi^N$, 
$$\mathcal{H} = \Bigg\{\delta \in \Xi^N\,\Bigg|\,\sum_{k=0}^{N-1}\,b_k\,h(\delta_k) \leq \bar{\mathcal{B}} \Bigg\}$$
we have that $\delta^*(N_{\oplus,1},N_{\ominus,1}) \in \mathcal{H} \wedge \delta^*(N_{\oplus,2},N_{\ominus,2}) \in \mathcal{H} \Rightarrow \hat{\delta}^{(t)} \in \mathcal{H}$. We examine two mutually exclusive situations.  
\begin{enumerate}
    \item $ \sum_{k=0}^{N-1}\,b_k\, h(\hat{\delta}^{(t)}_k) < \bar{\mathcal{B}}.$ Then, it is possible to decrease at least one entry of $\hat{\delta}^{(t)}$ and obtain a strictly better objective value meaning that $\delta^*(N_{\oplus,1},N_{\ominus,1})$ and $\delta^*(N_{\oplus,2},N_{\ominus,2})$ were not jointly optimal: contradiction.
    \vspace{5pt}
    \item $ \sum_{k=0}^{N-1}\,b_k\, h(\hat{\delta}^{(t)}_k) = \bar{\mathcal{B}}.$ In such scenario, the entire segment $[\delta^*(N_{\oplus,1},N_{\ominus,1}),\delta^*(N_{\oplus,2},N_{\ominus,2})]$ must be an edge of $\mathcal{H}$. We can quickly invoke continuity arguments suggesting the existence of multiple schedules on that segment sharing both a common number of $\delta_k^* = M\bar{\delta}$ (same $N_{\oplus}$ value) and a common number of $\delta_k^* = m\bar{\delta}$ (same $N_{\ominus}$ value). Yet by linearity their objective value would remain equal to $u^*$. As mentioned above, this is impossible because there exists at most one valid schedule $\delta^*(N_{\oplus},N_{\ominus})$ by pair $(N_{\oplus},N_{\ominus})$: contradiction.
\end{enumerate}
\end{proof}
\subsection{Closed-form schedules with \emph{cut-off}.}
\label{full_delta}
We extend the family of $\{h_{r>0}\}$ functions and incorporate the limiting case $r = 0$.\\ For any $r\geq0$,
\begin{equation}
    h_{r}(\delta) = \begin{cases} 
   -\log(\delta) & r=0 \\
    (\delta)^{-r} & r>0
    \end{cases} \hspace{5pt} \Rightarrow \hspace{5pt} (h_{r}')^{(-1)}\big(-\omega\big) = \begin{cases} 
  \big(\omega\big)^{-1}& r=0 \\
    \Big(\frac{\omega}{r}\Big)^{-\frac{1}{r+1}}  & r>0
    \end{cases}
    \label{eq:family_broad_extended}
\end{equation} 

We kindly remind that we only focus on instances of $h$ functions positive on their respective domains. Both Theorem \ref{corollary_h_r} and \ref{corollary_h_r_work} immediately follow from Theorem \ref{general_theorem}, their respective proof is therefore omitted.
\vspace{10pt}

\theorem
\label{corollary_h_r} 
Let Assumptions A and B hold with $(a,b) \in \mathbb{R}^{N \times 2}_{++}$, $m < 1 < M$ and $h=h_{r\geq0} : [m\,\bar{\delta},M\,\bar{\delta}] \to \mathbb{R}_+$. 

\vspace{5pt}
$\exists N_{\oplus},N_{\ominus} \in \{0,\dots,N-1\}$, $\hat{\lambda} \in \mathbb{R}$ such that $\forall k \in \{0,\dots,N-1\}$
\begin{equation}
 \delta_{k}^* = \begin{cases} M\,\bar{\delta} & k \in \oplus := \{ \tilde{k} \,|\, \rho(\tilde{k}) < N_{\oplus}\} \\ \hat{\lambda} \, \big(\frac{b_k}{a_k}\big)^{\frac{1}{r+1}} & k \in \mathcal{T} := \{\tilde{k}\,|\,N_{\oplus}\leq \rho(\tilde{k}) \leq N-1-N_{\ominus}\} \\ m\,\bar{\delta} & k\in \ominus := \{  \tilde{k} \,|\, \rho( \tilde{k}) > N-1-N_{\ominus}\} \end{cases} \label{eq:schedule_h_r}
\end{equation}
where $\rho(k)$ depicts the \emph{descending rank} of $\nu_k = b_k / a_k$, $\hat{\lambda} > 0$ is defined as
  \begin{equation}
  \hat{\lambda} = 
  \begin{cases} \bar{\delta} \cdot \Bigg[\frac{M^{\sum_{k \in \oplus}\,b_k}\cdot m^{\sum_{k \in \ominus}\,b_k}}{\Pi_{k\in \mathcal{T}}\,(b_k/a_k)^{-b_k}}\Bigg]^{-\big(\sum_{k\in \mathcal{T}}\,b_k\big)^{-1}} & r=0 \\
\bar{\delta} \cdot\Bigg[\frac{\big(\sum_{k=0}^{N-1}\,b_k\big)-\big(\sum_{k\in\oplus}\,b_k\big)\,M^{-r}-\big(\sum_{k\in\ominus}\,b_k\big)\,m^{-r}}{\sum_{k\in \mathcal{T}}\,(a_k^r\,b_k)^{\frac{1}{r+1}}}\Bigg]^{-\frac{1}{r}} &r>0 
\end{cases}
\label{eq:closed_hat_lambda}
\end{equation}
 and $\delta^*$ stands as a solution of \eqref{eq:master_problem}.
\vspace{15pt}
\theorem
\label{corollary_h_r_work} 
Let Assumptions A and B hold with $(a,b) \in \mathbb{R}^{N \times 2}_{++}$, $\omega_M < \bar{\omega} / N < \omega_m$ and $h=h_{r>0} : \mathbb{R}_+ \to \mathbb{R}_+$. 

\vspace{5pt}
$\exists N_{\oplus},N_{\ominus} \in \{0,\dots,N-1\}$, $\hat{\lambda} \in \mathbb{R}$ such that $\forall k \in \{0,\dots,N-1\}$
\begin{equation}
 \omega_{k}^* = \begin{cases} \omega_M & k \in \oplus := \{ \tilde{k} \,|\, \rho(\tilde{k}) < N_{\oplus}\} \\ \hat{\lambda}\,(a_k^r\,b_k)^{\frac{1}{r+1}} & k \in \mathcal{T} := \{\tilde{k}\,|\,N_{\oplus}\leq \rho(\tilde{k}) \leq N-1-N_{\ominus}\} \\ \omega_m & k\in \ominus := \{  \tilde{k} \,|\, \rho( \tilde{k}) > N-1-N_{\ominus}\} \end{cases} \label{eq:schedule_h_r_work}
\end{equation}
where $\rho(k)$ depicts the \emph{descending rank} of $\nu_k = (a_k^r\,b_k)^{-1}$, $\hat{\lambda} > 0$ is defined as
  \begin{equation}
\hat{\lambda} = \Bigg[\frac{\bar{\omega}-N_{\oplus}\,\omega_M - N_{\ominus}\,\omega_m}{\sum_{k\in \mathcal{T}}\,(a_k^r\,b_k)^{\frac{1}{r+1}}}\Bigg] \label{eq:two_work}
\end{equation}
 and $\omega^*$ stands as a solution of \eqref{eq:master_problem_manual}.
\end{document}